\documentclass[12pt]{amsart}

\usepackage{ishai}

\title[The polylog quotient and the Goncharov quotient]{The polylog quotient and the Goncharov quotient in computational Chabauty-Kim theory I}

\author{David Corwin, Ishai Dan-Cohen}

\date{\today}

\usepackage{fullpage}
\usepackage{colonequals}
\usepackage{microtype}
\begin{document}

\maketitle

\raggedbottom
\SelectTips{cm}{11}

\section*{Abstract}

\textit{Polylogarithms} are those \textit{multiple polylogarithms} that factor through a certain quotient of the de Rham fundamental group of the thrice punctured line known as the \emph{polylogarithmic quotient}. Building on work of Dan-Cohen, Wewers, and Brown, we push the computational boundary of our explicit motivic version of Kim's method in the case of the thrice punctured line over an open subscheme of $\Spec{\mathbb{Z}}$. To do so, we develop a greatly refined version of the algorithm of Dan-Cohen tailored specifically to this case, and we focus attention on the polylogarithmic quotient. This allows us to restrict our calculus with motivic iterated integrals to the so-called depth-one part of the mixed Tate Galois group studied extensively by Goncharov. We also discover an interesting consequence of the symmetry-breaking nature of the polylog quotient that forces us to symmetrize our polylogarithmic version of Kim's conjecture.

In this first part of a two-part series, we focus on a specific example, which allows us to verify an interesting new case of Kim's conjecture.

\tableofcontents

\section{Introduction}


\subsection{The Context for Our Work}

\subsubsection{The $S$-Unit Equation and Kim's Method}

For a ring $R$, the \emph{unit equation} asks for solutions to $x+y=1$ for which $x$ and $y$ are units in the ring $R$. If we let $X=\Poneminusthreepoints$, this is equivalent to finding elements of $X(R)$. For $R=\Oc_K[1/S]$ for a number field $K$ and a finite set $S$ of finite places, this is known as the \emph{$S$-unit equation}. The interest in the unit equation comes from the 1929 theorem of Siegel that the $S$-unit equation has finitely many solutions for fixed $K$ and $S$.

In 2004, Kim (\cite{kim05}) reproved this theorem in the case $K=\Qb$ by showing that there were nonzero Coleman functions on $X(\Zp)$ that vanish on $X(R)$, using a method reminiscent of those of Chabauty (\cite{Chabauty41}) and Skolem. Kim later extended this proof to totally real fields in \cite{KimTangential}. 
More specifically, for each $n \in \Zb_{>0}$, $Z=\Spec{R}$ and $\pfff \in Z$ lying over $p$, Kim's method produces an ideal $\Ic^Z_{n,\mathrm{Kim}}$ in the ring of Coleman functions on $X(Z_{\pfff})$, whose set $X(Z_{\pfff})_{n,\Kim}^Z$ of common zeroes contains $X(Z)$. Concretely, the functions so produced are polynomials in $p$-adic polylogarithms of degree (or ``half-weight'') at most $n$.

\subsubsection{Computational Motivic Chabauty-Kim Theory}

Since then, there has been work to compute specific elements of $\Ic^Z_{n,\mathrm{Kim}}$. More specifically, Dan-Cohen and Wewers (\cite{CKone,MTMUE}) developed a motivic framework for making computations and computed the cases $Z=\Spec{\Zb}, \Spec{\Zb[1/2]}$, both in half-weights $n=2,4$. In \cite{MTMUEII}, Dan-Cohen formulated these methods into an algorithm and proved that this algorithm outputs the set of solutions, assuming variants on some well-known conjectures. Brown (\cite{BrownIntegral}) made further contributions and put the computations of \cite{MTMUE} in a new light.

One motivation for the computations of \cite{MTMUE} is that they verified specific cases of the following conjecture of Kim:
\begin{conj}[Conjecture 3.1 of \cite{nabsd}]\label{conj:kim1}
$X(Z) = X(Z_{\pfff})_{n,\Kim}^Z$ for sufficiently large $n$.
\end{conj}

\subsection{The Current Work}

In this work, we present various simplifications to the methods of \cite{MTMUE} and \cite{MTMUEII}, inspired in part by \cite{BrownIntegral}, and use them to find an element of $\Ic^Z_{n,\mathrm{Kim}}$ for $Z=\Spec{\Zb[1/3]}$. More specifically, for every positive integer $n$, we define an ideal $\Ic^Z_{\PL,n}$ of elements of $\Oc(\Pi_{\PL,n} \times \pi_1^\un(Z))$ (defined in Section \ref{sec:PL}), whose elements specialize under a $p$-adic realization map to elements of $\Ic^Z_{n,\mathrm{Kim}}$. We then prove the following:
\begin{thm}[Theorem \ref{thm:computation_wt4}]
The element
\begin{align*}
\zetau(3) \logu(3) \Liu_4 - \left(\frac{18}{13} \Liu_4(3) - \frac{3}{52} \Liu_4(9)\right) \logu \Liu_3 \\
- \frac{(\logu)^3 \Liu_1}{24} \left(\zetau(3) \logu(3) - 4 \left(\frac{18}{13} \Liu_4(3) - \frac{3}{52} \Liu_4(9)\right)\right)
\end{align*}
 is in $\Ic^Z_{\PL,4}$ for $Z=\Spec{\Zb[1/3]}$.
\end{thm}

We also let $X(Z_{\pfff})_{\PL,n}$ denote the subset of common zeroes of specializations of elements of $\Ic^Z_{\PL,n}$. By computing values of the corresponding Coleman function at elements of $X(Z_{\pfff})_{\PL,2}$ already found in \cite{nabsd} and using a symmetrization argument described below, we get:

\begin{thm}[Theorem \ref{thm:verification_computation}]\label{thm:verification_computation1}
Conjecture \ref{conj:kim1} holds for $Z = \Spec{\Zb[1/3]}$ and $p=5,7$ with $n=4$.
\end{thm}

This example exhibits a new phenomenon: while the functions we obtain cut out the set of $\Zb[1/3]$-points, its coefficients are written in terms of periods that ramify at the prime $2$.

There are two differences between the definitions of $X(Z_{\pfff})_{n,\Kim}^Z$ and $X(Z_{\pfff})_{\PL,n}$. The first is a technical point, discussed in Remark \ref{rem:kimvsus}, that conjecturally has no effect. The second, a simplification already introduced in \cite{MTMUE} and \cite{MTMUEII}, is that we work only with polylogarithms rather than multiple polylogarithms, i.e. functions of the form $\Li_{n_1,\cdots,n_r}$ only for $r=1$ rather than arbitrary positive integers $r$. This entails a certain strengthening of Conjecture \ref{conj:kim1}. However, as the following theorem shows, polylogarithms in and of themselves are insufficient for this purpose; the precise formulation requires some care.

We always have $X(Z_{\pfff})_{n,\Kim}^Z \subseteq X(Z_{\pfff})_{\PL,n}$ (c.f. Remark \ref{rem:kimvsus}), and we prove:
\begin{thm}[Theorem \ref{thm:counterexample}]\label{thm:counterexample1}
For any prime $\ell$ and positive integer $n$, we have
\[
-1 \in X(Z_{\pfff})_{\PL,n},
\]
where $Z = \Spec{\Zb[1/\ell]}$.

\end{thm}

As $X(\Zb[1/\ell])=\emptyset$ for $\ell$ odd, Theorem \ref{thm:counterexample1} disproves a previous hope by the authors and others that $X(Z) = X(Z_{\pfff})_{\PL,n}$ for sufficiently large $n$. This seemed like a reasonable generalization of Conjecture \ref{conj:kim1}, since not only is $\Ic^Z_{n,\mathrm{Kim}}$ nonzero for sufficiently large $n$, but $\Izn$ is as well, and hence $X(Z_{\pfff})_{\PL,n}$ is finite for such $n$ (more strongly, the rank of $\Izn$ goes to $\infty$ as $n \to \infty$).
However, there is an action of $S_3$ on the scheme $X$, and we may use it to correct the problem posed by Theorem \ref{thm:counterexample1} and refine our polylogarithmic version of Kim's conjecture accordingly. We write
$$
X(Z_{\pfff})_{\PL,n}^{S_3} \colonequals \bigcap_{\sigma \in S_3} \sigma(X(Z_{\pfff})_{\PL,n}).$$

Our strengthened Kim's conjecture is the following:
\begin{conj}[Conjecture \ref{conj:polylog_kim_symm}]\label{conj:polylog_kim_symm1}
$X(Z) = X(Z_{\pfff})_{\PL,n}^{S_3}$ for sufficiently large $n$.
\end{conj}
As explained in Section \ref{sec:conjectures}, Conjecture \ref{conj:polylog_kim_symm1} implies Conjecture \ref{conj:kim1}. It is in fact this strengthened conjecture that we verify in Theorem \ref{thm:verification_computation1}.


\subsection{Background on Mixed Tate Motives and the Unit Equation}

In this part of the introduction, we assume familiarity with the basics of Kim's program (\cite{kim05},\cite{kim09})\footnote{See \cite{KimNonLinear}, \cite{KimFGDG}, \cite{KimGTDG}, and \cite{ChabautytoKim} for introductions to Kim's program.} and explain both the application of mixed Tate motives to the $S$-unit equation (as developed in \cite{MTMUE}) and the novelty in our methods. Our setup is then repeated in full technical detail in Section \ref{sec:technical_preliminaries}.

\subsubsection{From Kim's Method to Mixed Tate Motives}

For a place $\pf$ of $K$ over $p$ and a smooth curve $X$ over $Z$, Kim constructs a commutative diagram (\cite{kim05},\cite{kim09})
\[ 
\xymatrix{
X(Z) \ar@{^{(}->}[r]
\ar[d]_-{\kappa}
&
X(Z_{\pf})
\ar[d]^-{\kappa_{\pf}}
\\
\mathrm{Sel}(X/Z)_n \ar[r]_-{\mathrm{loc}_n}
&
\mathrm{Sel}(X/Z_{\pf})_n \, ,
}
\]
which we refer to as \emph{Kim's cutter}, and proves for $X=\Poneminusthreepoints$ that the morphism of schemes $\mathrm{loc}_n$ is non-dominant for sufficiently large $n$. The ideal $$\Ic^Z_{n,\mathrm{Kim}}$$ is then defined as the set of pullbacks under $\kappa_{\pf}$ of functions on $\mathrm{Sel}(X/Z_{\pf})_n$ vanishing on the image of $\mathrm{loc}_n$.

In order to compute such functions concretely, one must understand $\mathrm{Sel}(X/Z)_n$ as well as the morphism $\mathrm{loc}_n$. Let $U_n$ denote the $n$th quotient of the pro-unipotent completion of the \'etale fundamental group of $X_{\overline{K}}$ along the descending central series as in \cite{nabsd}. The Selmer variety is defined so that its set of $\Qp$-points is the set $$H^1_f(G_K;U_n)$$ of cohomology classes of $G_K$ with coefficients in $U_n$ that are unramified at closed points of $Z$ and crystalline at primes over $p$. Both the group $G_K$ and the local conditions are hard to understand explicitly.

An important observation is that one needs to understand only the \emph{category} of continuous $p$-adic representations of $G_K$ that appear in $U_n$ and its torsors. More specifically, $U_n$ and its torsors are pro-varieties over $\Qp$, and their coordinates rings are (ind-)objects of a certain subcategory of the category of all continuous $p$-adic representations of $G_K$.

This subcategory is the category of mixed Tate $p$-adic representations of $\Qp$ unramified at closed points of $Z$ and crystalline at places above $p$; being ``mixed Tate'' means that its semi-simplification is a direct sum of tensor powers $\Qp(n)$ for $n \in \Zb$ of the $p$-adic cyclotomic character. The subcategory of semisimple objects is therefore equivalent to the category of representations of $\Gm$, so the full category is equivalent by the Tannakian formalism to the category of representations of a group $\pi_1^{\MT}(Z)$ isomorphic to an extension of $\Gm$ by a pro-unipotent group. The pro-unipotent group may be determined by computing the Bloch-Kato Selmer groups $H^1_f(G_K,\Qp(n))$ for each $n$, and these are known (\cite{Soule79}). In this way, the Selmer variety becomes simply the group cohomology of $\pi_1^{\MT}(Z)$, with no further local conditions (other than those encoded in the category itself).

In fact, the category of mixed Tate Galois representations with local conditions mentioned above is just the extension of scalars from $\Qb$ to $\Qp$ of the category $$\MT(Z, \QQ)$$ of mixed Tate motives over $Z$ with coefficients in $\Qb$. This latter category was defined in \cite{DelGon05}, and its Tannakian fundamental group is denoted $\pi_1^{\MT}(Z)$. The unipotent de Rham fundamental group $\pi_1^\un(X)$ is the Tannakian fundamental group of the category of vector bundles with unipotent integrable connection, and the theory of \cite{DelGon05} (or the later theories of \cite{LevineTMFG} and \cite{DCSchlank17}) gives it an action of $\pi_1^{\MT}(Z)$. This motivic Selmer variety, as developed in \cite{HadianDuke} and \cite{MTMUE}, is just the group cohomology of $\pi_1^{\MT}(Z)$, with coefficients in (a quotient depending on $n$ of) $\pi_1^\un(X)$.


More specifically, let $\Pi$ denote a $\pi_1^{\MT}(Z)$-equivariant quotient of $\pi_1^\un(X)$. Then there is a motivic version of Kim's cutter:
\[ 
\xymatrix{
X(Z) \ar@{^{(}->}[r]
\ar[d]_-{\kappa}
&
X(Z_{\pf})
\ar[d]^-{\kappa_{\pf}}
\\
H^1(\pi_1^{\MT}(Z),\Pi) \ar[r]_-{\mathrm{loc}_\Pi}
&
\Pi(\Qp)
}.
\]
The computability of this diagram as opposed to Kim's original diagram comes from our precise understanding of the abstract structure of the group $\pi_1^{\MT}(Z)$ and Goncharov's study of certain special functions on it.

\subsubsection{Structure of the Category of Mixed Tate Motives}

More specifically, every object $M$ of the category $\MT(Z,\QQ)$ has an ascending filtration $W$ such that for all integers $n$,
\[
\mathrm{Gr}^W_{2n-1}M = 0,
\]
and
\[
\mathrm{Gr}^W_{2n}M
\]
is a direct sum of copies of $\Qb(-n)=\Qb(1)^{\otimes -n}$, where $\Qb(1)$ is known as the \emph{Tate motive} and has Galois realization the cyclotomic character. We will refer to $n$ as the \emph{half-weight}.

This implies that the subcategory of semisimple objects of $\MT(Z,\QQ)$ is equivalent to the category of representations of $\Gm$, so $\pi_1^{\MT}$ is the extension of $\Gm$ by a pro-unipotent group, which we call $\pi_1^{\un}(Z)$. By the Tannakian formalism, one may describe $\pi_1^{\un}(Z)$ in terms of the groups $\mathrm{Ext}^i(\QQ,\QQ(n))$. In fact, one has
\[
    \Ext^i(\QQ,\QQ(n)) = K_{2n-i}^{(n)}(Z) =  \left\{\begin{array}{lr}
        K_{2n-1}(Z)_{\Qb}, & \text{for } i=1,\, n \ge 0\\
        0, & \text{for } n<0 \\
        0, & \text{for } i>1
        \end{array}\right\}.
\]

The groups $K_{2n-1}(Z)_{\Qb}$ are known by the work of Borel (\cite{Borel74}). From now on, we restrict to the case $Z=\Zb[1/S]$. Then $K_1(Z)_{\Qb}$ has dimension $S$, and for $n \ge 2$, $K_{2n-1}(Z)_\Qb = K_{2n-1}(K)_\Qb$ has dimension $0$ for $n$ even and $1$ for $n$ odd. By the formalism of pro-unipotent groups, this implies that the Lie algebra $\nf(Z)$ of $\pi_1^{\un}(Z)$ is a free graded Lie algebra with generators we denote $\{\tau_{\ell}\}_{\ell \in S}$ and $\{\sigma_{2n+1}\}_{n \ge 1}$ in degrees $-1$ and $-2n-1$, respectively (the degree determines the action of $\Gm$). The coordinate ring of $\pi_1^{\un}(Z)$, which we denote by $A(Z)$, is then a free graded vector space on elements we denote $f_w$, for $w$ a word in the generators of $\nf(Z)$. The product and coproduct on $A(Z)$ are the shuffle product and deconcatenation coproduct, respectively, reviewed in Section \ref{sec:free_prounip}.

This means that we have a description of $\pi_1^{\MT}(Z)$ as an abstract (pro-algebraic) group, which allows us to determine the structure of $H^1(\pi_1^{\MT}(Z),\Pi)$, at least when we understand the action of $\pi_1^{\MT}(Z)$ on $\Pi$. However, we still need a concrete description that allows us to compute $\kappa$, $\mathrm{loc}_{\Pi}$, and $\kappa_{p}$ in coordinates.

\subsubsection{Motivic Periods}

It turns out that we may write concrete elements of $A(Z)$ as \emph{motivic polylogarithms} of the form $\mathrm{Li}_n^{\mathfrak{u}}(z)$ for $n \in \Zb_{\ge 1}$ and $z \in \Qb$, as well as \emph{motivic logarithms} $\logu(z)$ and \emph{motivic zeta values} $\zetau(n)$ for $n \in \Zb_{\ge 1}$ (Section \ref{sec:gen_az}). There is also a ring homomorphism
\[\mathrm{per}_{p} \colon A(Z) \to \Qp\]
sending $\Liu_n(z)$ to the Coleman (\cite{Coleman}) $p$-adic polylogarithm $\Lip_n(z)$, and this is expected to be injective (c.f. Conjecture \ref{conj:period_conj}). There is also an explicit formula for the (reduced) coproduct of these elements in the Hopf algebra $A(Z)$, due to Goncharov (\cite{GonGal}):
\[
\Delta'\Liu_n(z) \colonequals \Delta \Liu_n(z) - 1 \otimes \Liu_n(z) - \Liu_n(z) \otimes 1 =
 \sum_{i=1}^{n-1}  \Liu_{n-i}(z) \otimes \frac{(\logu(z))^{\Sha i}}{i!}.
\]

To explicitly understand the maps in (motivic) Kim's cutter, we must say a bit more about $H^1(\pi_1^{\MT}(Z),\Pi)$ and $\Pi$. To understand $H^1(\pi_1^{\MT}(Z),\Pi)$, we define for any $\Qb$-algebra $R$:
\[
\ZPi(R) := Z^1(\pi_1^\un(Z)_R, \Pi_R)^{\Gm},
\]
the set of $\Gm$-equivariant algebraic cocycles from $\pi_1^\un(Z)_R$ to $\Pi_R$. By \cite[Proposition 5.2.1]{MTMUE}, we have
\[
H^1(\pi_1^{\MT}(Z), \Pi) = \ZPi(\Qb).
\]

For $\Pi$, we focus on finite-dimensional quotients of the \emph{polylogarithmic fundamental group of $X$}, defined as follows. The group $\pi_1^{\un}(X)$ is free pro-unipotent in two generators $e_0$ and $e_1$. Hence, $\Oc(\pi_1^{\un}(X))$ has a canonical shuffle basis parametrized by words $w$ in $e_0$ and $e_1$; we denote its elements by $\Liu_w$. We set $\logu = \Liu_{e_0}$ and $\Liu_m = \Liu_{e_1 e_0 \cdots e_0}$, and $\Pi_{\PL,n} = \pi_1^{\PL}(X)_{\ge -n}$ is the quotient corresponding to the subalgebra generated by $\logu,\Liu_1,\cdots,\Liu_n$. 

With this terminology, it is easy to express $\kappa$, $\mathrm{loc}_{\Pi}$, and $\kappa_{p}$ explicitly (c.f. Section \ref{sec:kappa_in_coords}). For $z \in X(Z)$, $\kappa(z)$ is the cocycle whose induced map on coordinate rings is
\begin{align*}
\Liu_n \mapsto \Liu_n(z)\\
\logu \mapsto \logu(z).
\end{align*}
For a cocycle $c \in \ZPi(\Qb)$, we let $c^\sharp \colon \Oc(\pi_1^{\PL}(X)_{\ge -n}) \to A(Z)$ denote the associated homomorphism of algebras. Then $\mathrm{loc}_{\pi_1^{\PL}(X)_{\ge -n}}(c) = \mathrm{loc}_{\PL,n}(c) \in \pi_1^{\PL}(X)_{\ge -n}(\Qp)$ is given by
\[
\mathrm{per}_{p} \circ c^\sharp.
\]
Finally, for $z \in X(Z_{p})$, $\kappa_{p}(z) \in \pi_1^{\PL}(X)_{\ge -n}(\Qp)$ is given by
\begin{align*}
\Liu_n \mapsto \Lip_n(z)\\
\logu \mapsto \logp(z).
\end{align*}

\subsubsection{The Geometric Step}\label{sec:the_geometric_step}

One important point, already present in \cite{MTMUEII} and \cite{BrownIntegral}, is that $\mathrm{loc}_n$, despite its apparent $p$-adic origin, is actually defined over the rationals. More specifically, we have a \emph{universal cocycle evaluation map} (c.f. Definition \ref{defn:universal_cocycle}): $$\evalu_\Pi \colon \ZPi \times \pi_1^\un(Z) \to \Pi \times \pi_1^\un(Z).$$
Letting $\Kc$ denote the function field of $\pi_1^\un(Z)$, we also have the base change from $\pi_1^\un(Z)$ to $\Spec{\Kc}$
$$
\evalu_\Pi^{\Kc} \colon (\ZPi)_{\Kc} \to \Pi_{\Kc}.
$$
We then let $\IzPi$ denote the ideal of functions in the coordinate ring of $\Pi_{\Kc}$ vanishing on the image of 
$\evalu_\Pi^{\Kc}$. This is known as the \emph{Chabauty-Kim ideal} (associated to $\Pi$). For $\Pi=\Pi_{\PL,n}$, we denote it by $\Izn$.

If $\Pi$ is finite dimensional, it will turn out that $\evalu_\Pi^{\Kc}$ is a morphism of affine spaces over the field $\Kc$, and our ``geometric step'' (c.f. Section \ref{sec:geometric_step}) consists in computing its scheme-theoretic image. The convenience of this is that it can be done in an abstract basis $\{f_w\}$ of $A(Z)$ and depends only on the size of $S$. One particularly useful tool in the geometric step is a description of a coordinate system on $\ZPLn$ (Proposition \ref{prop:brown}) based on an idea communicated to the authors by Francis Brown. This enables us to circumvent the computation of the exponential map, as in 1.10 and 4.2.4 of \cite{MTMUEII}.

The results of the geometric step look like:
\[\Liu_2 - \frac{1}{2} \logu \Liu_1\]
\[f_{\sigma_3} f_{\tau_{\ell}} \Liu_4 - f_{{\sigma_3} {\tau_{\ell}}} \logu \Liu_3 - \frac{(\logu)^3 \Liu_1}{24} \left(f_{\sigma_3} f_{\tau_{\ell}} - 4 f_{{\sigma_3} {\tau_{\ell}}}\right),\]
as elements of $\Ic^Z_{\PL,4}$ for $Z = \Spec{\Zb[1/\ell]}$.

\subsubsection{Computing Bases of $A(Z)$ in Low Degrees}

To find the Coleman functions associated to elements of $\Ic^Z_{\PL,n}$, we must apply $\mathrm{per}_{p}$ to the coefficients. To do this for abstract elements (e.g., $f_{\sigma_3 \tau_{\ell}}$), we must write elements $f_w \in A(Z)$ in terms of concrete motivic periods of the form $\Liu_n(z)$, $\zetau(n)$, or $\logu(z)$. This is done for $Z=\Spec{\Zb[1/2]}$ and $Z=\Spec{\Zb[1/3]}$ up to degree $4$ in Section \ref{sec:galois_coord_ex}, and we now summarize how we do it.


We first shrink $Z$ if necessary to a smaller open subscheme $Z'$, so that $X(Z')$ has enough elements to generate $A(Z')$ in the degrees we want. For $Z=\Spec{\Zb[1/2]}$, this turns out to be unnecessary, but for $Z=\Spec{\Zb[1/3]}$, we must shrink to $Z' = \Spec{\Zb[1/6]}$ as in done in \ref{sec:choosing_p_3_1/6}.

We then study $A_n(Z')$ inductively in $n$. For $A_1(Z')$, this is easy, since $\logu(\ell)$ corresponds to $f_{\tau_\ell}$. In higher degrees, we use the reduced coproduct, which we may compute using Goncharov's formula, to reduce to the case of lower degrees. Reduced coproducts, however, can give relations between $\Liu_n(z)$'s only modulo the kernel of $\Delta'$, which by Proposition \ref{prop:exact_sequence} is generated as a rational vector space in degree $2n+1$ by $\zetau(2n+1)=f_{\sigma_{2n+1}}$ and is zero in even degree. The only general method we know for determining the rational multiple of $\zetau(2n+1)$ is to apply $\mathrm{per}_{p}$ for some prime $p$ and approximate the rational number numerically using \cite{Lip}.


The Coleman functions we want are polynomials in $p$-adic polylogarithms whose coefficients are themselves polynomials with rational coefficients in special values of $p$-adic polylogarithms. In general, this method gives us only $p$-adic approximations of these rational coefficients. Nevertheless, approximations still allow us to $p$-adically approximate the roots and thereby verify cases of Kim's conjecture. Furthermore, in specific instances, one may use functional equations to verify the desired identities between $p$-adic polylogarithms, as is done in the Appendix to \cite{MTMUE} and the Appendix (Section \ref{sec:appendix}) to this paper.

We note that these computations are similar to those done in the various articles of Zagier (\cite{Zagier91}, \cite{ZagierAppendix}, \cite{GanglZagier}, etc.), and in one case (c.f. Remark \ref{rem:ident_already}), one of our identities already appears in a paper of Gangl and Zagier. One important difference, however, is that our computations take place in the Hopf algebra $A(Z)$, rather than in the Bloch groups as in Zagier's works. In particular, we must worry about products of lower-degree polylogarithms.




\subsection{Notation}

For a scheme $Y$, we let $\Oc(Y)$ denote its coordinate ring. If $R$ is a ring, we let $Y \otimes R$ or $Y_R$ denote the product (or `base-change') $Y \times \Spec{R}$. If $Y$ and $\Spec{R}$ are over an implicit base scheme $S$ (often $\Spec{\Qb}$), we take the product over $S$. Similarly, if $M$ is a linear object (such as a module, an algebra, a Lie algebra, or a Hopf algebra), then $M_R$ denotes $M \otimes R$ (again, with the tensor product taken over an implicit base ring).

If $f \colon Y \to Z$ is a morphism of schemes, we denote by $f^{\#} \colon \Oc(Z) \to \Oc(Y)$ the corresponding homomorphism of rings. Similarly, if $\alpha \in Y(R)$, we have a homomorphism $\alpha^\# \colon \Oc(Y) \to R$.

\subsection{Acknowledgements}\label{sec:acknowledgements}

The first author would like to thank his thesis advisor Bjorn Poonen for numerous comments and corrections on this paper. He would like to thank Amnon Besser for various pieces of advice, especially for a trick for going between complex and $p$-adic identities in Proposition \ref{prop:p_adic_ident}. He would like to thank Rob de Jeu, for useful discussion, eventually leading him to the paper \cite{Zagier91}, as well as Don Zagier, for subsequent discussions. He would like to thank Herbert Gangl for the idea of proving Lemma \ref{lemma:complex_ident} by substituting $x=-1$ and $y=1/3$ into the Kummer-Spence identity.

The first author was supported by NSF grants DMS-1069236 and DMS-160194, Simons Foundation grant \#402472, and NSF RTG Grant \#1646385 during various parts of the writing of this paper.

Both authors would like to thank Francis Brown for discussions about the subject and especially for the idea behind Proposition \ref{prop:brown}.

\section{Technical Preliminaries}\label{sec:technical_preliminaries}

This paper builds on the work of \cite{MTMUE}, \cite{MTMUEII}, and \cite{BrownIntegral}. We recall some of the important objects in the theory.

\subsection{Generalities on Graded Pro-Unipotent Groups}\label{sec:prounip}

\subsubsection{Conventions for Graded Vector Spaces}

Our definition of graded vector space is the following:

\begin{defn}\label{defn:graded_vector_space}
A graded vector space is a collection of vector spaces $V_i$ indexed by $i \in \Zb$.
\end{defn}

\begin{defn}A graded vector space is positive (respectively negative, strictly positive, strictly negative) if $V_i = 0$ for $i<0$ (respectively, for $i>0$, for $i \le 0$, for $i \ge 0$).\end{defn}

In general, we will consider only graded vector spaces satisfying one of these four conditions.
Furthermore, \emph{unless otherwise stated, we will consider only graded vector spaces $\{V_i\}$ such that each $V_i$ is finite-dimensional,} which ensures that the double dual is the identity. One must then be careful when taking tensor constructions to ensure that the result of the construction still satisfies this finite-dimensionality condition, as follows. Specifically, we consider the tensor product only between two graded vector spaces if they are either both positive or both negative, and we consider the tensor algebra only of a strictly positive or strictly negative graded vector space.

\begin{rem}\label{rem:graded}
Traditionally, a graded vector space is thought of as an underlying vector space with extra structure, rather than a collection of vector spaces as in Definition \ref{defn:graded_vector_space}. For a collection $$V = \{V_i\}_{i \in \Zb}$$ of (finite-dimensional) vector spaces indexed by the integers, we may take either the direct sum $$V^{\bigoplus} \colonequals \bigoplus_{i} V_i$$ or the direct product $$V^{\prod} \colonequals \prod_i V_i$$ as our underlying vector space. In general, we use the former for coordinate rings and Lie coalgebras and the latter for universal enveloping algebras and Lie algebras. As all of our pro-unipotent groups will be negatively graded, we use the $\bigoplus$ notion for positively graded vector spaces and the $\prod$ notion for negatively graded vector spaces.

When considering the $\prod$ notion, we take completed tensor product instead of tensor product (and similarly for tensor algebras and universal enveloping algebras), and a coproduct is a complete coproduct, i.e., a homomorphism $V \to V \widehat{\otimes} V$. In addition, a set of homogeneous elements is considered a basis if it generates $V$ in each degree, and a similar remark applies to bases of algebras and Lie algebras. When taking the dual of a negatively graded vector space, we take the graded dual and then view the resulting (positively) graded vector space via the $\bigoplus$ notion. In particular, the double dual is always the original vector space. Nonetheless, we have the following relation between graded and ordinary (non-graded) duals: $$(V^{\bigoplus})^{\vee} = (V^{\vee})^{\prod}.$$

\end{rem}

\subsubsection{Graded Pro-unipotent Groups}

Let $U$ be a pro-unipotent group over $\Qb$.
Then $U$ is a group scheme over $\Qb$, so its coordinate ring $\Oc(U)$ is a Hopf algebra over $\Qb$. We recall that if $\Oc(U)$ is a Hopf algebra over $\Qb$, then it is equipped with a product $\Oc(U) \otimes \Oc(U) \to \Oc(U)$, a coproduct $\Delta \colon \Oc(U) \to \Oc(U) \otimes \Oc(U)$, a unit $\eta \colon \Qb \to \Oc(U)$, and a counit $\epsilon \colon \Oc(U) \to \Qb$. The kernel $I(U)$ of $\epsilon$ is known as the \emph{augmentation ideal}. We also write $\Delta'(x) \colonequals \Delta(x) - x\otimes 1 - 1 \otimes x$ for the \emph{reduced coproduct}. Finally, we say that an element is \emph{primitive}\footnote{Note that this is unrelated to the term ``primitive non-extension'' in \cite[1.6]{MTMUEII}.} if it is in the kernel of $\Delta'$.

We say that a Hopf algebra $A$ is a \emph{graded Hopf algebra} if the multiplication $A \otimes A \to A$, the coproduct $A \to A \otimes A$, unit $\Qb \to A$, and counit $A \to \Qb$, are morphisms of graded vector spaces, where $A \otimes A$ has the standard grading on a tensor product, and $\Qb$ is in degree zero.

\begin{defn}\label{defn:graded_hopf}
By a \emph{grading} on $U$, we mean a positive grading on $\Oc(U)$ as a $\Qb$-vector space such that the degree zero part of $\Oc(U)$ is one-dimensional over $\Qb$.
\end{defn}

\begin{defn}\label{defn:delta_n}
If $A$ is a Hopf algebra graded in the sense of Definition \ref{defn:graded_hopf}, we let $\Delta_n$ and $\Delta'_n$ denote the restrictions of $\Delta$ and $\Delta'$, respectively, to $A_n$, the $n$th graded piece.
\end{defn}

Furthermore, $\Delta_n$ and $\Delta'_n$ map $A_n$ into the $n$th graded piece of $A \otimes A$, which is
\[
\bigoplus_{i+j=n} A_i \otimes A_j.
\]

\begin{defn}\label{defn:delta_ij}
For $i+j=n$ and $i,j \ge 0$, we let $\Delta_{i,j}$ and $\Delta'_{i,j}$ denote the projections of $\Delta_n$ and $\Delta'_n$, respectively, to $A_i \otimes A_j$. One may check via the axioms defining a Hopf algebra that $\Delta'_{i,j} = 0$ when either of $i$ or $j$ is zero.
\end{defn}

The reduced coproduct $\Delta'$ induces a (graded) Lie coalgebra structure on $I(U)/I(U)^2$, and the dual Lie algebra $(I(U)/I(U)^2)^{\vee}$ is the Lie algebra $\nf$ of $U$. It is a strictly negatively graded pro-nilpotent Lie algebra. We let $\Uc U = \Uc \nf$ denote the dual Hopf algebra of $\Oc(U)$, which is the (completed) universal enveloping algebra of $\nf$. The composition
\[
\mathrm{ker}(\Delta') \hookrightarrow \Uc U = \Oc(U)^{\vee} \twoheadrightarrow I(U)^{\vee}
\]
induces an isomorphism between the set of primitive elements of $\Uc U$ and the Lie algebra $\nf = (I(U)/I(U)^2)^{\vee} \subseteq I(U)^{\vee}$.

Furthermore, for a $\Qb$-algebra $R$, we may identify $U(R)$ with the group of grouplike elements in $(\Uc U)_R$, i.e., $x$ such that
\[
\Delta x = x \otimes x.
\]
Evaluation of an an element of $\Oc(U)$ on an element of $U(R)$ is given by evaluation on the corresponding grouplike element of $(\Uc U)_R$.


The functor sending $U$ to $\nf$ is known to be an equivalence of categories between graded pro-unipotent groups and strictly negatively graded pro-nilpotent Lie algebras. For each positive integer $n$, the set of elements $\nf_{<-n}$ is a Lie ideal, and we denote by $\nf_{\ge -n}$ the quotient $\nf/\nf_{<-n}$. We denote the corresponding quotient pro-unipotent group by $$U_{\ge -n},$$
and it is a unipotent algebraic group. In fact, $U$ is the inverse limit
\[
\varprojlim_n U_{\ge -n}.
\]

\begin{ex}
If $\nf$ is a one-dimensional Lie algebra generated by an element $x$, then $\nf$ is nilpotent. We have $\Uc \nf = \Qb[[x]]$, and $\Oc(U) = \Qb[f_x]$, with both $x$ and $f_x$ primitive. Then $U$ is the group $\Gb_a$, and the set of grouplike elements of $\Uc \nf$ is the set of elements of the form $\exp(rx)$ for $r \in \Qb$. In particular, this demonstrates the usefulness of taking completed universal enveloping algebras.
\end{ex}

\subsubsection{Free Pro-unipotent Groups}\label{sec:free_prounip}

\begin{defn}\label{defn:free_prounip}
If $V$ is a strictly negative graded vector space, we may form the free pro-nilpotent Lie algebra on $V$ as follows. We take the graded tensor algebra $TV$ on $V$ and put the unique coproduct on it such that all elements of $V$ are primitive. The subspace of primitive elements of $TV$ forms a graded pro-nilpotent\footnote{The Lie algebra is pro-nilpotent rather than free by Remark \ref{rem:graded}, because it is negatively graded.} Lie algebra, denoted $\nf(V)$, with corresponding pro-unipotent group denoted $U(V)$. Then $\nf(V)$, $U(V)$ are known as the \emph{free pro-nilpotent Lie algebra} and \emph{free pro-unipotent group}, respectively, on the graded vector space $V$.\end{defn}

\begin{rem}\label{rem:adjoint}
The construction $V \mapsto \nf(V)$ is left adjoint to the forgetful functor from graded pro-nilpotent Lie algebras to graded vector spaces.
\end{rem}

\begin{defn} If $I$ is an index set with a degree function $d \colon I \to \Zb_{<0}$ with finite fibers, then the free vector space $V^I$ on the set $I$ naturally obtains the structure of a negatively graded
In this case, the \emph{free pro-unipotent group on the set $I$} is just the free pro-unipotent group on $V^I$.
\end{defn}

In particular, the free pro-unipotent group on a graded vector space is isomorphic to the free pro-unipotent group on a (graded) basis of that vector space. The Lie algebra is the pro-nilpotent completion of the usual (non-graded) free Lie algebra on the set $I$ and as such is generated by the elements of $I$.

The graded dual Hopf algebra of $TV = \Uc \nf(V)$ is the coordinate ring $\Oc(U(V))$. Let $\{x_i\}$ be a graded basis of $V$, so that words $w$ in the $\{x_i\}$ form a basis of $\Uc \nf(V)$, and let $\{f_w\}_w$ denote the basis of $\Oc(U(V))$ dual to $\{w\}$. Then $\Oc(U(V))$ is isomorphic to the \emph{free shuffle algebra} on the graded vector space $V^{\vee}$. Its coproduct is known as the \emph{deconcatenation coproduct} and is given by
\[
\Delta f_w \colonequals \sum_{w_1 w_2 = w} f_{w_1} \otimes f_{w_2},
\]and the (commutative) product $\Sha$ on $\Oc(U(V))$, known as the \emph{shuffle product}, is given by:

$$
f_{w_1} \Sha f_{w_2} \colonequals \sum_{\sigma \in \Sha(\ell(w_1),\ell(w_2))} \sigma(f_{w_1 w_2})
,$$
where $\ell$ denotes the length of a word, $\Sha(\ell(w_1),\ell(w_2)) \subseteq S_{\ell(w_1)+\ell(w_2)}$ denotes the group of shuffle permutations of type $(\ell(w_1),\ell(w_2))$, and $w_1 w_2$ denotes concatenation of words.

\begin{rem}\label{rem:single_letter_prim}
It follows from the definition of the deconcatenation coproduct that a word consisting of a single letter is a primitive element of the free shuffle algebra.
\end{rem}

\subsubsection{Conventions for Products}

Let $\alpha$ and $\beta$ be two paths in a space $X$. Then in the literature, the sybmol $\alpha \beta$ can have two different meanings. It can either denote:
\begin{enumerate}[(i)]
    \item\label{lexical} The path given by going along $\alpha$ and then $\beta$
    \item\label{functional} The path given by going along $\beta$ and then $\alpha$
\end{enumerate}

The first is known as the `lexical' order and the second as the `functional' order. In this paper, we will use the lexical order, in contrast to the convention of \cite{MTMUE} (but consistent with \cite{BrownIntegral}). However, we would like to take a moment to explain how these two conventions help clarify differing conventions in the literature for iterated integrals, multiple zeta values, polylogarithms, and more. We will refer to these differing conventions as the lexical and functional conventions, respectively.

In fact, either convention necessitates a particular convention for iterated integrals. In general, one wants a ``coproduct'' formula to hold for iterated integrals (c.f. \cite[\S 2.1]{BrownDecomp}, \cite[5.1(ii)]{HainClassical}, or \cite[Proposition 5]{HainNotes}), by which we mean
\begin{equation}\label{eqn:int_coprod}
\int_{\alpha \beta} \omega_1 \cdots \omega_n = \sum_{i=0}^n \int_{\alpha} \omega_1 \cdots \omega_i \int_{\beta} \omega_{i+1} \cdots \omega_n
\end{equation}

In order for the coproduct to take this nice form (\ref{eqn:int_coprod}), our convention for path composition determines our convention for iterated integrals. More specifically, those that use the lexical order for path composition use the formula

\[
I(\gamma(0);\omega_1,\cdots,\omega_n;\gamma(1)) \colonequals \int_\gamma \omega_1 \cdots \omega_n = \int_{0 \le t_1 \le \cdots \le t_n \le 1} f_1(t_1) \cdots f_n(t_n) dt_1 \cdots dt_n,
\]
where $\gamma^*(\omega_i)=f_i(t) dt$,
and those that use the functional order for path composition use the formula

\[
I(\gamma(0);\omega_1,\cdots,\omega_n;\gamma(1)) \colonequals \int_\gamma \omega_1 \cdots \omega_n = \int_{0 \le t_n \le \cdots \le t_1 \le 1} f_1(t_1) \cdots f_n(t_n) dt_1 \cdots dt_n.
\]

Given that these conventions are opposite, the corresponding conventions for the iterated integral expression for polylogarithms must be opposite. More precisely, the iterated integral defining a multiple polylogarithm must always begin with $\frac{dz}{1-z}$ in the lexical convention, while it must always end with $\frac{dz}{1-z}$ in the functional convention. More precisely, let us define
\begin{equation}\label{eqn:polylog_series}
\Li_{s_1,\cdots,s_r}(z) \colonequals \sum_{0<k_1<\cdots<k_r} \frac{z^{k_r}}{k_1^{s_1} \cdots k_r^{s_r}}.
\end{equation}
Set $e^0 = \frac{dz}{z}$ and $e^1 = \frac{dz}{1-z}$. Then using the lexical convention for iterated integration, we have:
$$
\Li_{s_1,\cdots,s_r}(z) = I(0;e^1, \underbrace{e^0, \cdots, e^0}_{s_1-1}, e^1, \underbrace{e^0, \cdots, e^0}_{s_2-1}, \cdots, e^1, \underbrace{e^0, \cdots, e^0}_{s_r-1};z).
$$

In fact, the definition itself of $\Li_{s_1,\cdots,s_r}(z)$ depends on the convention. More specifically, if we were to use the functional convention for iterated integrals in tandem with (\ref{eqn:polylog_series}), we would get:
$$
\Li_{s_1,\cdots,s_r}(z) = I(0; \underbrace{e^0, \cdots, e^0}_{s_r-1}, e^1, \underbrace{e^0, \cdots, e^0}_{s_{r-1}-1}, \cdots, e^1, \underbrace{e^0, \cdots, e^0}_{s_1-1}, e^1;z).
$$

This is precisely the formula that appears in \cite[(1.4)]{BrownElliptic}. However, most authors prefer the $s_i$'s to appear in the iterated integral in the same order as they do in the argument of the function. Therefore, almost all papers that use the functional convention for iterated integration will write:

$$
\Li_{s_1,\cdots,s_r}(z) \colonequals \sum_{k_1>\cdots>k_r>0} \frac{z^{k_r}}{k_1^{s_1} \cdots k_r^{s_r}}.
$$

As a result, one then writes:
$$
\Li_{s_1,\cdots,s_r}(z) = I(0; \underbrace{e^0, \cdots, e^0}_{s_1-1}, e^1, \underbrace{e^0, \cdots, e^0}_{s_{2}-1}, \cdots, e^1, \underbrace{e^0, \cdots, e^0}_{s_r-1}, e^1;z).
$$

Thus, the convention one uses for path composition determines the convention one uses for $\Li_{s_1,\cdots,s_r}(z)$ (except in \cite{BrownElliptic}). Similarly, the two conventions for multiple zeta values follow this paradigm. Specifically, those who use the lexical convention write $$\zeta(s_1,\cdots,s_r) = \sum_{k_1>\cdots>k_r>0} \frac{1}{k_1^{s_1} \cdots k_r^{s_r}},$$ and those who use the functional convention write $$\zeta(s_1,\cdots,s_r) = \sum_{k_1>\cdots>k_r>0} \frac{1}{k_1^{s_1} \cdots k_r^{s_r}}.$$

Note, however, that $\Li_n$ always denotes the same function (both as a multi-valued complex analytic function, a Coleman function, and an abstract function on the de Rham fundamental group), no matter which convention one uses. In fact, this brings us back to the two conventions for path composition. The fact that some write $\Liu_n$ (c.f. Section \ref{fund_grp_coords}) as $e^1 \underbrace{e^0 \cdots e^0}_{n-1}$ and others write it as $\underbrace{e^0 \cdots e^0}_{n-1} e^1$, yet both denote the exact same regular function on the unipotent de Rham fundamental group of $\Poneminusthreepoints$, is due to the differing conventions for path composition.

As we use the lexical convention, one will find the coproduct formula
\begin{equation}\label{eqn:polylog_coprod}
\Delta'\Liu_n = 
 \sum_{i=1}^{n-1} \Liu_{n-i} \otimes  \frac{(\logu)^{\Sha i}}{i!}
 \end{equation}
in this paper (\ref{prop:polylog_coprod}). With the other convention, one must write \[\Delta'\Liu_n = \sum_{i=1}^{n-1} \frac{(\logu)^{\Sha i}}{i!} \otimes \Liu_{n-i}.\]
 
More subtly, these conventions affect the convention one uses for the motivic coproduct. More specifically, if we use the lexical convention, we want to also be able to write
\begin{equation}\label{eqn:motivic_polylog_coprod}
\Delta'\Liu_n(z) = 
 \sum_{i=1}^{n-1} \Liu_{n-i}(z) \otimes  \frac{(\logu(z))^{\Sha i}}{i!}
\end{equation}
rather than its opposite. This formula is correct as long as we use the lexical order for composition in $\pi_1^\un(Z)$ (which, in particular, comes out in how one writes the Goncharov coproduct of \cite{GonGal}; Goncharov uses the lexical order himself). Theoretically, one could use one convention for composition in $\pi_1^\un(X)$ and another for composition in $\pi_1^\un(Z)$, but that would cause formulas (\ref{eqn:polylog_coprod}) and (\ref{eqn:motivic_polylog_coprod}) to conflict with each other.

One important implication of the difference in formulas for the motivic coproduct is:

\begin{rem}\label{rem:different_convention_13}
Our $f_{\sigma \tau}$ is actually the $\phi_{1.3}$ of \cite[7.6.1]{MTMUE}, even though the notation would suggest it is $\phi_{3.1}$.
\end{rem}

In terms of authors and sources, one may find the lexical convention and/or the other conventions that go along with it in \cite[\S 1.2]{BrownICM}, \cite[\S 5]{HainClassical}, \cite[Definition 2]{HainNotes}, \cite[(9.1)]{BrownIntegral}, \cite[Definition 1.1]{hain87}, \cite[Definition 2.1]{BrownDecomp}, \cite{GonGal}, \cite[(1.1.1)]{ChenIterated}, \cite[(0.1)]{FurushoI}, \cite[(2)]{GonRegulators}, \cite[(1), (2), (4), and Definition 1.2]{GonMPMTM}, \cite[\S 11]{GonICM}, \cite{ZagierModForm}, \cite{ZagierValues}, \cite{TerasomaMTM}, and \cite{rabi}. One may also find conventions for multiple zeta values consistent with this convention in other articles by Francis Brown, such as \cite{BrownMTZ}, \cite{BrownDepthGraded}, \cite{BrownICM}, and \cite{BrownSingle}.

On the other hand, in \cite[1.16]{MTMUE}, \cite[5.16.1]{DelGon05}, \cite[(79)]{CartierBourbaki}, \cite[(0.1) and 5.1A]{DelMuN}, \cite{deligne13}, \cite{HoffmanMHS}, \cite{RacinetDouble}, \cite{SouderesMotDouble}, \cite[2.2.4]{MTMUEII}, and \cite{brown04}, one finds the functional conventions.

\subsection{The Various Fundamental Groups}

\subsubsection{The Mixed Tate Fundamental Group of $Z$}\label{sec:mixed_tate_fund_grp}

\begin{defn}\label{defn:open_integer_scheme}
An \emph{open integer scheme} is an open subscheme of $\Spec \Oo_K$, where $K$ is a number field and $\Oo_K$ its ring of integers. 
\end{defn}

Let $Z \subseteq \Spec \Oo_K$ be an open integer scheme, $\MT(Z, \QQ)$ its Tannakian category of mixed Tate motives with $\QQ$-coefficients, which exists by \cite{DelGon05}\footnote{It is constructed by putting a t-structure due to \cite{Levine93} on a certain subcategory of Voevodsky's triangulated category $DM_{gm}(K)$ of \cite{Voevodsky00}, taking the heart of that t-structure, and finally taking the subcategory of objects with good reduction at closed points of $Z$. The category $DM_{gm}(K)$ is defined by taking a certain localization of the category of complexes of smooth varieties over a field with correspondences as morphisms, then taking its pseudo-abelian envelope, and finally inverting the Tate object $\Qb(1)$. However, we will need only the properties of $\MT(Z, \QQ)$, not its construction.}. This is a category with realization functors
\[
\mathrm{real}^\sigma \colon \MT(Z, \QQ) \to \mathrm{MHS}_{\Qb}
\]
to mixed Hodge structures with $\Qb$-coefficients for each embedding $\sigma \colon K \hookrightarrow \Cb$ and
\[
\mathrm{real}^\ell \colon \MT(Z, \QQ) \to \mathrm{Rep}_{\Ql}(G_K)
\]
to $\ell$-adic representations of $G_K$ for each prime $\ell$. The image of each realization functor consists of mixed Tate objects, i.e., objects with a composition series consisting of tensor powers of the image of the Tate object $\Qb(1) \colonequals h_2(\Pb^1;\Qb)$.

\begin{defn}\label{defn:good_reduction}
A continuous $\Qb_{\ell}$-representation of $G_K$ for a number field $K$ is said to have \emph{good reduction} at a non-archimedean place $v$ of $K$ if either $v \nmid \ell$, and the representation is unramified at $v$, or if $v \mid \ell$, and the representation is crystalline at $v$.
\end{defn}

The $\ell$-adic realizations of an object of $\MT(Z, \QQ)$ form a compatible system of $\Qb_{\ell}$-Galois representations with good reduction at closed points of $Z$ (in particular, crystalline at primes dividing $\ell$). If $(X,D)$ is a pair of a scheme and codimension $1$ subscheme, both smooth and proper over $Z$ and rationally connected, then the relative cohomology $h^*(X,D;\Qb)$ is an object of this category such that
\[
\mathrm{real}^\sigma(h^*(X,D;\Qb)) = H^*_{\mathrm{Betti}}(X_\sigma^{\mathrm{an}}(\Cb),D_{\sigma}^{\mathrm{an}}(\Cb);\Qb),
\]
\[
\mathrm{real}^\ell(h^*(X,D;\Qb)) = H^*_{\mathrm{\acute{e}t}}(X_{\overline{K}},D_{\overline{K}};\Ql),
\]
with their associated mixed Hodge structure and continuous $G_K$-action, respectively.

The only simple objects of $\MT(Z, \QQ)$ are the objects $\Qb(n) \colonequals \Qb(1)^{\otimes n}$ for $n \in \Zb$, each object has a finite composition series, and the extensions are determined by the fact that
\[
\Ext^1(\Qb(0),\Qb(n)) = K_{2n-1}(Z)_\Qb
\]
\[
\Ext^i = 0 \hspace{1cm} \forall \, i \ge 2.
\]

The groups $K_{2n-1}(Z)_{\Qb}$ are known by the work of Borel (\cite{Borel74}). For $n=1$, we have $K_1(Z) = \Oc(Z)^\times$, and for $n \ge 2$, $K_{2n-1}(Z)_\Qb = K_{2n-1}(K)_\Qb$ has dimension $r_2$ for $n$ even and $r_1+r_2$ for $n$ odd, where $r_1$ and $r_2$ are the numbers of real and complex places of $K$, respectively.

\begin{defn}\label{defn:canonical_fiber_functor}
Let $M$ be an object of $\MT(Z, \QQ)$. Then $M$ has an increasing filtration $W_i M$ known as the \emph{weight filtration}. The quotient $W_i M / W_{i-1} M$ is trivial when $i$ is odd and is isomorphic to a direct sum of copies of $\Qb(-i/2)$ when $i$ is even. We let
\[
\mathrm{Can}(M) \colonequals \bigoplus_{i \in \Zb} \mathrm{Hom}_{\MT(Z, \QQ)}(\Qb(-i), W_{2i} M / W_{2i-1} M),
\]
and we call this the \emph{canonical fiber functor}.
\end{defn}

Then the category $\MT(Z, \QQ)$ is a neutral $\Qb$-linear Tannakian category with fiber functor $\mathrm{Can}$, and we let $\pi_1^{\MT}(Z)$ denote its Tannakian fundamental group, which is therefore a pro-algebraic group over $\Qb$.

The subcategory of simple objects of $\MT(Z, \QQ)$ consists of direct sums of tensor powers of $\Qb(1)$ and is therefore equivalent as a Tannakian category to the category of representations of $\Gm$. This inclusion induces a quotient map $\pi_1^{\MT}(Z) \twoheadrightarrow \Gm$, and we let $\pi_1^\un(Z)$ denote the kernel of this quotient. The functor sending an object $M$ of $\MT(Z, \QQ)$ to the direct sum $\bigoplus_{i \in \Zb} W_i M / W_{i-1} M$ gives a splitting of this inclusion of categories, which implies that the quotient map $\pi_1^{\MT}(Z) \twoheadrightarrow \Gm$ splits.

This implies that $\MT(Z, \QQ)$ has fundamental group
\[
\pi_1^{\MT}(Z) = \pi_1^\un(Z) \rtimes \Gm,
\]
where $\pi_1^\un(Z)$ is the maximal pro-unipotent subgroup of $\pi_1^{\MT}(Z)$. The action of $\Gm$ on $\pi_1^\un(Z)$ by conjugation gives an action of $\Gm$, or equivalently, a grading, on the Hopf algebra of $\pi_1^\un(Z)$. This associated graded Hopf algebra is denoted
\[
\bigoplus_{i=0}^\infty A(Z)_i = A(Z) := \Oo(\pi_1^\un(Z)),
\]
where $A(Z)_i$ denotes the $n$th graded piece. We refer to the degree on $A(Z)$ as the \emph{half-weight}, since it is half the ordinary motivic weight.

In fact, the description of the $\mathrm{Ext}$ groups gives us the following information. It gives a canonical embedding of graded vector spaces $$\bigoplus_{n=1}^\infty K_{2n-1}(Z)_\Qb \hookrightarrow A(Z),$$
with $K_{2n-1}(Z)_{\Qb}$ in degree $n$, and whose image is the set of primitive elements of $A(Z)$. Equivalently, this gives a canonical isomorphism $$\pi_1^\un(Z)^{\mathrm{ab}} \cong \left(\bigoplus_{n=1}^\infty K_{2n-1}(Z)_\Qb \right)^{\vee}.$$ In fact, this canonical isomorphism extends to an isomorphism between $\pi_1^\un(Z)$ and the free pro-unipotent group (Definition \ref{defn:free_prounip}) on the graded vector space $\left(\bigoplus_{n=1}^\infty K_{2n-1}(Z)_\Qb \right)^{\vee}$, with $(K_{2n-1}(Z))^{\vee}$ in degree $-n$, but this extension is not canonical. This last fact tells us that there is a non-canonical isomorphism between $A(Z)$ and the free shuffle algebra on the graded vector space $\bigoplus_{n=1}^\infty K_{2n-1}(Z)_\Qb$, with $K_{2n-1}(Z)$ in degree $n$. This non-canonicity is the key to a later consideration; see Remark \ref{rem:canonical2}.

Furthermore, it is not hard to show that $\Ext^1(\Qb(0),\Qb(n))$ is isomorphic to the space of degree $n$ primitive elements of $A(Z)$; see Proposition \ref{prop:exact_sequence} below.

For $Z' \subseteq Z$ an open subscheme, we have an inclusion $\MT(Z, \QQ) \subseteq \MT(Z', \QQ)$, which gives rise to a quotient map $\pi_1^{\MT}(Z') \twoheadrightarrow \pi_1^{\MT}(Z)$ that is an isomorphism on $\Gm$ and hence to an inclusion
$$
A(Z) \subseteq A(Z')
$$
of graded Hopf algebras. There is also a graded Hopf algebra $A(\Spec{K})$, which is the union of $A(Z)$ for $Z \subseteq \Spec{\Oc_K}$, and we may view all such $A(Z)$ as lying inside $A(\Spec{K})$. If $\pf$ is a point of $\Spec{\Oc_K} \setminus Z$ and $\alpha \in A(Z)$, we say $\alpha$ is \emph{unramified at $\pf$} if $$\alpha \in A(Z \cup \{\pf\}) \subseteq A(Z).$$

\subsubsection{The de Rham Unipotent Fundamental Group of $X$}
\
For the rest of this paper, we let $X=\Poneminusthreepoints$ over $\mathbb{Z}$.


\begin{defn}\label{defn:de_rham_fund_grp}
For $x$ an element of $X(\Qb)$ or a rational tangential basepoint (as defined in \cite[\S 15]{Deligne89}), we let \[\pi_1(X;x)\] denote the unipotent de Rham fundamental group of $X_{\Qb}$ at $x$. It is the fundamental group of the Tannakian category of algebraic vector bundles with nilpotent connection on $X_{\Qb}$ with fiber functor the fiber at $x$. For $x=\vec{1}_0$, the vector $1$ at the point $0$, we denote it simply by
\[
\pi_1^\un(X).
\]
\end{defn}

This pro-unipotent group over $\Qb$ is a free pro-unipotent group on the graded vector space consisting of $H^{\mathrm{dR}}_1(X_{\Qb})$ in degree $-1$ (and zero in other degrees). Dually, its coordinate ring is generated by words in holomorphic differential forms on $X$, which are integrands of interated integrals.

By the construction in \cite[\S 3]{DelGon05}, $\pi_1^\un(X)$ is in $\MT(Z, \QQ)$ (in the sense that its coordinate ring and Lie algebra are each an Ind-object and pro-object, respectively, of $\MT(Z, \QQ)$). It therefore carries an action of $\pi_1^{\MT}(Z)$, whose restriction to $\Gb_m$ induces the grading.

More generally, for any two rational basepoints $a,b$ of $X$ (tangential or ordinary), there is an object $_bP_a$\footnote{As \cite{DelGon05} uses the functional convention for path composition, this is in fact the torsor of paths from $a$ to $b$.} in $\MT(Z, \QQ)$ (again by \cite[\S 3]{DelGon05}). It is a torsor on the left under $\pi_1(X;b)$ and on the right under $\pi_1(X;a)$, and it reduces to $\pi_1(X;x)$ when $a=b=x$. There is a canonical $\Qb$-point of the torsor $_bP_a$, known as the canonical de Rham path, defined as the unique point in $F^0{_bP_a}$. This implies that as a group, $\pi_1(X;x)$ does not depend on the basepoint $x$. However, the canonical point is not in general fixed under the action of $\pi_1^{\MT}(Z)$, and thus the $\pi_1^{\MT}(Z)$-action does depend on $x$.


\begin{rem}\label{rem:independence_of_basepoint}
By an argument analogous to the proof of \cite[Corollary 2.9]{nabsd}, one may check that all of these constructions (in particular, the Chabauty-Kim locus, c.f., Section \ref{sec:CK_locus}) are the same if we replace $\vec{1}_0$ by any other $Z$-integral basepoint of $X$.
\end{rem}

\begin{rem}
We will often consider $\pi_1^{\MT}(Z)$-equivariant quotients $\pi_1^\un(X) \twoheadrightarrow \Pi$, especially when $\Pi$ is finite-dimensional as a scheme over $\Qb$ (hence an algebraic group). \emph{Unless otherwise stated, it is always understood that $\Pi$ is such a quotient.}

A standard example is $\Pi_n \colonequals \pi_1^\un(X)_{\ge -n}$. However, in most of our calculations, we will be concerned with quotients that factor through $\pi_1^\PL(X)$, whose precise definition we recall in Section \ref{sec:PL}.
\end{rem}

We write $\nf(Z)$, $\nf(X)$, and $\nf^\PL(X)$ for the corresponding Lie algebras.

\subsection{Cohomology and Cocycles}

We let $H^1(\pi_1^{\MT}(Z), \pi_1^\un(X))$ denote the (pointed) set of $\pi_1^{\MT}(Z)$-equivariant torsor schemes under $\pi_1^\un(X)$.


For $b \in X(Z)$, the torsor $_b P_{\vec{1}_0}$ is a $\pi_1^{\MT}(Z)$-equivariant torsor under $_{\vec{1}_0} P_{\vec{1}_0} = \pi_1^\un(X)$, hence an element of $H^1(\pi_1^{\MT}(Z), \pi_1^\un(X))$.

We therefore have the Kummer map
\[
X(Z) \xrightarrow{\kappa} 
H^1(\pi_1^{\MT}(Z), \pi_1^\un(X)),
\]
and by composition with the map induced by $\pi_1^\un(X) \twoheadrightarrow \Pi$,
\[
X(Z) \xrightarrow{\kappa} 
H^1(\pi_1^{\MT}(Z), \Pi).
\]

\begin{defn}\label{defn:cocycles}
For any $\Qb$-algebra $R$, we define
\[
\ZPi(R) := Z^1(\pi_1^\un(Z)_R, \Pi_R)^{\Gm},
\]
which is the set of morphisms of schemes from $\pi_1^\un(Z)_R$ to $\Pi_R$ over $R$, equivariant with respect to the $\Gm$-action, and satisfying the cocycle condition on the $R'$-points for any $R$-algebra $R'$.
\end{defn}

\cite[Proposition 6.4]{BrownIntegral} ensures that this is representable by a scheme (also see Corollary \ref{cor:psi_isom}). We thus write
\[
\ZPi = \ZPi(Z) := Z^1(\pi_1^\un(Z), \Pi)^{\Gm}
\]
for the scheme of $\Gm$-equivariant cocycles. If $\Pi$ is finite-dimensional, then this is in fact a finite-dimensional variety over $\Qb$.

By \cite[Proposition 5.2.1]{MTMUE}, we have
\[
H^1(\pi_1^{\MT}(Z), \Pi) = \ZPi(\Qb).
\]

\begin{defn}\label{defn:universal_cocycle}
The \emph{universal cocycle evaluation map} $$\evalu_\Pi \colon \ZPi \times \pi_1^\un(Z) \to \Pi \times \pi_1^\un(Z)$$ is defined on the functors of points as follows. For a $\Qb$-algebra $R$ and an element $(c,\gamma) \in (\ZPi)_R(R) \times \pi_1^\un(Z)_R(R) = (\ZPi \times \pi_1^\un(Z))(R)$, we have $c(\gamma) \in \Pi_R(R) = \Pi(R)$. We define $\evalu_\Pi(c,\gamma)$ to be the pair $(c(\gamma),\gamma)$.
\end{defn}


In fact, the morphism $\evalu_\Pi$ lies over the identity morphism on $\pi_1^\un(Z)$, so letting $\Kc$ denote the function field of $\pi_1^\un(Z)$, we also have the base change from $\pi_1^\un(Z)$ to $\Spec{\Kc}$
$$
\evalu_\Pi^{\Kc} \colon (\ZPi)_{\Kc} \to \Pi_{\Kc}.
$$

\begin{rem}\label{rem:geometric_step}
The advantage of working over $\Kc$ rather than $\pi_1^\un(Z)$ is that if $\Pi$ is finite dimensional, then $\evalu_\Pi^{\Kc}$ is a morphism of affine spaces over the field $\Kc$, so one may use elimination theory to compute its scheme-theoretic image. This computation is our ``geometric step'' (see Section \ref{sec:geometric_step}).
\end{rem}

\begin{defn}\label{defn:CK_ideal}
We let $\IzPi$ denote the ideal of functions in the coordinate ring of $\Pi_{\Kc}$ vanishing on the image of 
$\evalu_\Pi^{\Kc}$. This is known as the \emph{Chabauty-Kim ideal} (associated to $\Pi$). For $\Pi=\Pi_n$, we denote it by $\Ic^Z_n$.
\end{defn}


\subsection{\texorpdfstring{$\pf$}{p}-adic Realization and Kim's Cutter}

Let $\pf$ be a closed point of $Z$. For simplicity, we suppose that $Z_{\pf} \cong \Spec{\Zp}$.

If $\omega \in \Oc(\pi_1^\un(X))$, and $a,b$ are $\Zp$-basepoints of $X$ (rational or tangential), then we can extract an element of $\Qp$ known as the Coleman iterated integral $\int_a^b \omega$. The Coleman iterated integral is originally due to Coleman (\cite{Coleman}) and was reformulated by Besser \cite{Besser} into the form that is used in \cite{MTMUE}.

\begin{defn}\label{defn:local_kummer}
Fix $a = \vec{1}_0$, and let $b$ vary over $X(Z_{\pf})$. We define the \emph{local Kummer map}
$$
X(Z_{\pf}) \xrightarrow{\kappa_{\pf}} \pi_1^\un(X)(\Qp)
$$
by sending a regular function $\omega$ on $\pi_1^\un(X)$ to the Coleman function $\kappa_{\pf}^{\sharp}(\omega) \colon X(Z_{\pf}) \to \Qp$ defined by $b \mapsto \int_a^b \omega$. Its composition with $\pi_1^\un(X)(\Qp) \rightarrow \Pi(\Qp)$ is also denoted by $\kappa_{\pf}$.
\end{defn}

The local Kummer map is Coleman-analytic, meaning that regular functions on $\pi_1^\un(X)$ pull back to Coleman functions on $X(Z_{\pf})$. These are locally analytic functions, and a nonzero such function has finitely many zeroes.

\begin{rem}
The local Kummer map in this form was originally referred to in \cite{kim05} as the \emph{$p$-adic unipotent Albanese map}. It is the same as the map $\alpha$ of \cite[1.3]{MTMUE}. The latter is defined by sending $b$ to the torsor of paths from $a$ to $b$ (with its structure as a filtered $\phi$-module)
\end{rem}

\subsubsection{$\pf$-adic Period Map}


In addition, there is a morphism $\Spec{\Qp} \to \pi_1^\un(Z)$ (given by the map $\eta$ of \cite{ChatUnv} for $x=\pf$), or equivalently a $\Qb$-algebra homomorphism $$\mathrm{per}_{\pf} \colon A(Z) \to \Qp,$$ known as the \emph{$\pf$-adic period map}. While elements of $A(Z)$ are represented by formal (motivic) iterated integrals, this homomorphism takes the value of the iterated integral in the sense of Coleman integration.

The following may be regarded as a $p$-adic analogue of a small piece of the Kontsevich-Zagier period conjecture (\cite{KontZag}). It has been in folklore for some time and appears in the literature for $K/\Qb$ abelian as \cite[Conjecture 4]{Yamashita10}.

\begin{conj}[$p$-adic Period Conjecture]\label{conj:period_conj}
For any open integer scheme $Z$, the period map $\mathrm{per}_{\pf} \colon A(Z) \to \Qp$ is injective.
\end{conj}

\subsubsection{Kim's Cutter}

Viewing $\evalu_\Pi$ as a morphism of schemes over $\pi_1^\un(Z)$ and base-changing along the $\pf$-adic period map $\Spec{\Qp} \to \pi_1^\un(Z)$, we get a morphism $(\ZPi)_{\Qp} \to \Pi_{\Qp}$. The induced map on $\Qp$-points is denoted by $\mathrm{loc}_\Pi$.
Denoting the composition $X(Z) \xrightarrow{\kappa} H^1(\pi_1^{\MT}(Z), \Pi) = \ZPi(\Qb) \subseteq \ZPi(\Qp)$ by $\kappa$ as well, this fits into a diagram:

\[ 
\xymatrix{
X(Z) \ar@{^{(}->}[r]
\ar[d]_-{\kappa}
&
X(Z_{\pf})
\ar[d]^-{\kappa_{\pf}}
\\
\ZPi(\Qp) \ar[r]_-{\mathrm{loc}_\Pi}
&
\Pi(\Qp)
},
\]
which we call \emph{Kim's Cutter}.

This diagram is commutative (c.f. \cite{MTMUE}, 4.9). In Section \ref{sec:kappa_in_coords}, we will describe $\kappa$ and $\kappa_{\pf}$ explicitly in terms of coordinates.

\subsection{Chabauty-Kim Locus}\label{sec:CK_locus}

\begin{defn}\label{defn:restriction_to_padic_points}
Let $f \in \Oc(\Pi \times \pi_1^\un(Z))$. Now $\mathrm{per}_{\pf}$ induces a morphism $\Pi_{\Qp} \to \Pi \times \pi_1^\un(Z)$, and $f$ pulls back via this morphism to an element of the coordinate ring of $\Pi_{\Qp}$, hence a function $\Pi(\Qp) \xrightarrow{f_{\Qp}} \Qp$. The composition $$f_{\Qp} \circ \kappa_{\pf}$$ with $X(Z_{\pf}) \xrightarrow{\kappa_{\pf}} \Pi(\Qp)$ is a Coleman function on $X(Z_{\pf})$ and is denoted $f\restriction_{X(Z_{\pf})}$.
\end{defn}

\begin{defn}\label{defn:CK_locus}
We then define the \emph{Chabauty-Kim locus}
$$X(Z_{\pf})_{\Pi} = X(Z_{\pf})^Z_\Pi \colonequals \{z \in X(Z_{\pf}) \colon f\restriction_{X(Z_{\pf})}(z)=0 \hspace{0.5cm} \forall \,  f \in \Oc(\Pi \times \pi_1^\un(Z)) \cap \IzPi \}.
$$
\end{defn}

While this set depends on $Z$, we write $X(Z_{\pf})_{\Pi}$ when there is no confusion.

Any $f \in \Oc(\Pi \times \pi_1^\un(Z)) \cap \IzPi$ vanishes on the image of $\evalu_\Pi \colon \ZPi \times \pi_1^\un(Z) \to \Pi \times \pi_1^\un(Z)$, hence also on the image of $\mathrm{loc}_{\Pi}$. By the commutativity of Kim's Cutter, $f\restriction_{X(Z_{\pf})}$ vanishes on $X(Z)$, hence
\[
X(Z) \subset X(Z_{\pf})_\Pi.
\]

We note that if $\Pi'$ dominates $\Pi$ (i.e., we have $\pi_1^\un(X) \twoheadrightarrow \Pi' \xtwoheadrightarrow{p} \Pi$), then $p^\#(\IzPi) \subseteq \Ic^Z_{\Pi'}$, hence $X(Z_{\pf})_{\Pi'} \subseteq X(Z_{\pf})_\Pi$.

For $\Pi=\Pi_n$, we denote $X(Z_{\pf})_\Pi$ by $X(Z_{\pf})_n$. Kim (\cite{kim05,KimTangential}) showed that $X(Z_{\pf})_n$ is finite for sufficiently large $n$ when $K$ is totally real.

\begin{rem}\label{rem:kimvsus}
The locus originally defined by Kim and used in \cite[p.371]{nabsd}, which we denoted $X(Z_{\pf})_{n,\Kim}^Z$ in the introduction, differs slightly from our own in that it uses a version of $\IzPi$ defined as the ideal of functions vanishing on the image of the base change of $\evalu_\Pi$ along $\mathrm{per}_{\pf}$. That ideal could in principle be larger than our own, producing a possibly smaller locus. However, as explained in \cite[4.2.6]{MTMUEII}, Conjecture \ref{conj:period_conj} implies that these two are the same, so we expect the same results when doing it this way. Furthermore, our version of Kim's conjecture (Conjecture \ref{conj:kim}) is a priori stronger than the original version (Conjecture \ref{conj:kim1}) because
\[
X(Z_{\pf})_{n,\Kim}^Z \subseteq X(Z_{\pf})_n,
\]
so our theorems apply to his conjecture either way.
\end{rem}

\subsection{The Polylogarithmic Quotient}\label{sec:PL}

We let $N$ denote the kernel of the homomorphism $\pi_1^\un(X) \to \pi_1^\un(\Gm)$ induced by the inclusion $X \hookrightarrow \Gm$, where $\pi_1^\un(\Gm)$ refers to the unipotent de Rham fundamental group of $({\Gm})_{\Qb}$.

\begin{defn}

Following \cite{Deligne89}, we define the \emph{polylogarithmic quotient}
\[
\pi_1^\PL(X) \colonequals \pi_1^\un(X)/[N,N].
\]
\end{defn}

The group $\pi_1^{\MT}(Z)$ acts on $\pi_1^\un(\Gm)$ as well as $\pi_1^\un(X)$, and because $\pi_1^\un(X) \to \pi_1^\un(\Gm)$ is induced by a map of schemes over $\mathbb{Z}$, it is $\pi_1^{\MT}(Z)$-equivariant. Therefore, $N$ and hence $[N,N]$ are $\pi_1^{\MT}(Z)$-stable, so $\pi_1^\PL(X)$ has a structure of a motive, i.e., an action of $\pi_1^{\MT}(Z)$.

As a motive, it has the structure
\[
\pi_1^\PL(X) = \QQ(1) \ltimes \prod_{i=1}^\infty \QQ(i)
\]
(c.f.\cite[16]{Deligne89} and \cite[6]{DelGon05}) so in particular the action of $\pi_1^{\MT}(Z)$ factors through $\Gm$. 

\begin{rem}In Section \ref{fund_grp_coords}, we will describe $\pi_1^\PL(X)$ more explicitly in coordinates.
\end{rem}

In the specific case $\Pi = \Pi_{\PL,n} \colonequals \pi_1^\PL(X)_{\ge -n}$ for a positive integer $n$, we write $\ZPLn$, $\evalu_{\PL,n}$, $\loc_{\PL,n}$, $\Izn$, and $X(Z_{\pf})_{\PL,n}$ to denote $\ZPi$, $\evalu_\Pi$, $\loc_\Pi$, $\IzPi$, and $X(Z_{\pf})_\Pi$, respectively.

Since $\pi_1^\un(Z)$ acts trivially on $\pi_1^\PL(X)$, a $\Gm$-equivariant cocycle 
\[
c \colon \pi_1^\un(Z) \to \pi_1^\PL(X)
\]
is just a $\Gm$-equivariant homomorphism.

If $\Pi$ is a quotient of $\Pi_{\PL,n}$, then any $\Gm$-equivariant homomorphism $\pi_1^\un(Z) \to \Pi$ must be zero on $\pi_1^\un(Z)_{<-n}$ by the $\Gm$-equivariance, so $\ZPi$ is the same as $$Z^1(\pi_1^\un(Z)_{\ge -n}, \Pi)^{\Gm}.$$ It follows that we can in fact view $\evalu_{\PL,n}$ as a morphism $$\ZPLn \times \pi_1^\un(Z)_{\ge -n} \to \Pi_{\PL,n} \times \pi_1^\un(Z)_{\ge -n}$$ lying over the identity on $\pi_1^\un(Z)_{\ge -n}$. Letting $\Kn$ denote the function field of $\pi_1^\un(Z)_{\ge -n}$, it induces a map
$$
\evalu_{\PL,n}^{\Kn} \colon (\ZPLn)_{\Kn} \to (\Pi_{\PL,n})_{\Kn}
$$
of finite-dimensional affine spaces over the field $\Kn$, and we view $\Izn$ as an ideal in $\Oc((\Pi_{\PL,n})_{\Kn})$.

\subsection{Kim's Conjecture}\label{sec:conjectures}


Recall that $X(Z_{\pf})_n$ denotes the Chabauty-Kim locus associated to the quotient $\Pi_n = \pi_1^{\un}(X)_{\ge -n}$, and that by Remark \ref{rem:kimvsus}, it contains and is expected to be equal to $X(Z_{\pf})_{n,\Kim}^Z$. In the introduction, we stated Conjecture \ref{conj:kim1} from \cite{nabsd}, which uses $X(Z_{\pf})_{n,\Kim}^Z$, but we now state the variant of this conjecture that we will focus on:

\begin{conj}\label{conj:kim}
$X(Z) = X(Z_{\pf})_n$ for sufficiently large $n$.
\end{conj}

As explained in Remark \ref{rem:kimvsus}, this variant implies Conjecture \ref{conj:kim1}, and is implied by Conjecture \ref{conj:kim1} and Conjecture \ref{conj:period_conj} together. From now on, we focus entirely on Conjecture \ref{conj:kim}.

In fact, dimension counts show that $\Izn$ is nonzero; hence $X(Z_{\pf})_{\PL,n}$ is finite for sufficiently large $n$. To make the conjecture more computationally accessible, one would like an analogue of Conjecture \ref{conj:kim} for the polylogarithmic quotient. As explained in the introduction, one might therefore hope that $X(Z) = X(Z_{\pf})_{\PL,n}$ for sufficiently large $n$, which would be a strengthening of Conjecture \ref{conj:kim}. However, as we will show in Section \ref{sec:symmetrization}, this is not the case (at least for $Z = \Spec{\Zb[1/\ell]}$ and odd primes $\ell$). Nonetheless, we show how to get around this difficulty while still using only polylogarithms, by using a certain $S_3$-symmetrization.

More specifically, there is an action of $S_3$ on the scheme $X$. We then write
$$
X(Z_{\pf})_{\PL,n}^{S_3} \colonequals \bigcap_{\sigma \in S_3} \sigma(X(Z_{\pf})_{\PL,n}).$$

Our symmetrized conjecture is that
\begin{conj}\label{conj:polylog_kim_symm}
$X(Z) = X(Z_{\pf})_{\PL,n}^{S_3}$ for sufficiently large $n$.
\end{conj}

To relate this to Conjecture \ref{conj:kim}, we prove:

\begin{prop}\label{prop:different_conjs}
Conjecture \ref{conj:polylog_kim_symm} implies Conjecture \ref{conj:kim} (and hence by Remark \ref{rem:kimvsus}, Conjecture \ref{conj:kim1} as well).
\end{prop}

\begin{proof}
The $S_3$-action on $X$ gives a $\pi_1^{\MT}(Z)$-equivariant isomorphism from the de Rham fundamental group with basepoint $\vec{1}_0$ to the de Rham fundamental group with basepoint $\sigma(\vec{1}_0)$, which descends to an isomorphism on $\pi_1^{\un}(X)_{\ge -n}$. It follows by independence of basepoint (Remark \ref{rem:independence_of_basepoint}) and functoriality of all the constructions that $\sigma(X(Z_{\pf})_n)=X(Z_{\pf})_n$, so by $X(Z_{\pf})_n \subseteq X(Z_{\pf})_{\PL,n}$, we have
\[
X(Z_{\pf})_n \subseteq X(Z_{\pf})_{\PL,n}^{S_3},
\]
from which the result follows.
\end{proof}

In Theorem \ref{thm:verification_computation}, we use our computations to verify Conjecture \ref{conj:polylog_kim_symm} for $Z=\Spec{\Zb[1/3]}$ and $p=5,7$.

\section{Coordinates}

\subsection{Coordinates on the Fundamental Group}\label{fund_grp_coords}

Let $\{e_0,e_1\}$ be a basis of $H_1^{\mathrm{dR}}(X_{\Qb})$ dual to the basis $\left\{\frac{dz}{z},\frac{dz}{1-z}\right\}$ of $H^1_{\mathrm{dR}}(X_{\Qb})$. The algebra $\Uc \pi_1^\un(X)$ is the free (completed) non-commutative algebra on the generators $e_0$ and $e_1$, with coproduct given by declaring that $e_0$ and $e_1$ are primitive and grading given by putting both in degree $-1$. We refer to the words
\[
e_0, e_1, e_1e_0, e_1 e_0 e_0, \dots
\]
as the \emph{polylogarithmic words}. We let the elements $$\logu, \Liu_1, \Liu_2, \dots \in \Oo(\pi_1^\un(X))$$ be the duals of these words with respect to the standard basis of $\Uc \pi_1^\un(X)$. More generally, for a word $w$ in $e_0,e_1$, we let $\Liu_{w}$ denote the dual basis element of $\Oc(\pi_1^\un(X))$.

\begin{prop}\label{prop:polylog_coprod}

We have
\[
\Delta'\logu = 0
\]
\[
\Delta'\Liu_n = 
 \sum_{i=1}^{n-1}  \Liu_{n-i} \otimes \frac{(\logu)^{\Sha i}}{i!}.
\]

\end{prop}

\begin{proof}
The first equation follows from Remark \ref{rem:single_letter_prim}.

By the discussion following Definition \ref{defn:free_prounip}, we have the formula
\[
\Delta'\Liu_n = \Delta'\Liu_{e_1 e_0 \cdots e_0} = \Liu_{e_1} \otimes \Liu_{e_0 \cdots e_0} + \Liu_{e_1 e_0} \otimes \Liu_{e_0 \cdots e_0} + \cdots + \Liu_{e_1 e_0 \cdots e_0} \otimes \Liu_{e_0}.
\]

By the definition of the shuffle product, we have the formula
\[
(\logu)^{\sha i} = (\Liu_{e_0})^{\sha i} = i! \Liu_{(e_0)^i}.
\]
The previous two formulas together with the definition of $\Liu_n$ then imply the proposition.

\end{proof}

Letting $\PPL$ be the subalgebra generated by $\logu, \Liu_1, \Liu_2, \dots$, the formula above implies that $\PPL$ is a Hopf subalgebra. It therefore corresponds to a quotient group of $\pi_1^\un(X)$. It follows from \cite[Proposition 7.1.3]{MTMUE} that this quotient group is the group $\pi_1^\PL(X)$.

Furthermore, we have $$\Oo(\Pi_{\PL,n}) = \QQ[\logu, \Liu_1, \dots, \Liu_n]$$ as a Hopf subalgebra of $\PPL = \Oo(\pi_1^\PL(X))$.

Given a cocycle $c \in \ZPL$, we write $c^\sharp \colon \PPL \to A(Z)$ for the associated homomorphism of $\Qb$-algebras, and we write
\[
\logu(c) \colonequals c^\sharp \logu,
\]
\[
\Liu_n(c) \colonequals c^\sharp \Liu_n.
\]

\begin{cor}\label{cor:motivic_polylog_coprod}
For $c \in \ZPL(\Qb)$, we have
\[
\Delta'\logu(c) = 0
\]
\[
\Delta'\Liu_n(c) = 
 \sum_{i=1}^{n-1}  \Liu_{n-i}(c) \otimes \frac{(\logu(c))^{\Sha i}}{i!}.
\]
\end{cor}

\begin{proof}
The cocycle condition reduces to the homomorphism conditions (as we are restricting ourselves to the polylogarithmic quotient), which means, dually, that $c^\sharp$ respects the coproduct. The corollary then follows immediately from Proposition \ref{prop:polylog_coprod}.
\end{proof}

\subsection{Generating \texorpdfstring{$A(Z)$}{A(Z)}}\label{sec:gen_az}

We take a moment to discuss coordinates for the Galois group $\pi_1^\un(Z)$. On an abstract level, this is a free unipotent group, and its structure is governed by the theory of Section \ref{sec:prounip}. However, we need to be able to write elements of $A(Z)$ in a way that allows us to compute their $p$-adic periods. This essentially means writing them as explicit combinations of special values of polylogarithms and zeta functions.

We introduce the notation
\[
\logu(z) \colonequals \logu(\kappa(z)) \in A(Z)_1
\]
\[
\Liu_n(z) \colonequals \Liu_n(\kappa(z)) \in A(Z)_n
\]
for $z \in X(Z)$. It is the same as the motivic period $\Liu_n(z)$ mentioned in \cite[(9.1)]{BrownIntegral} and \cite[2.2.4]{MTMUEII}, which justifies the notation $\Liu_n(c)$ in the previous section.

Because $\logu$ is pulled back from $\Gm$, we in fact have $\logu(z) \in A(Z)$ whenever $z \in \Gm(Z)$.

We also note:
\begin{fact}\label{fact:functional_equations}
For $z,w \in X(Z)$, we have
\[
\logu(zw) = \logu(z) + \logu(w)
\]
and
\[
\Liu_1(z) = -\logu(1-z).
\]
\end{fact}

Letting $c_1$ denote the cocycle corresponding to the class of the path torsor $_{\vec{-1}_1} P_{\vec{1}_0}$ in $H^1(\pi_1^{\MT}(Z), \pi_1^\un(X))$, we write
\[
\zetau(n) \colonequals \Liu_n(c_1) \in A(Z)_n.
\]

It is not clear a priori that the special values $\Liu_n(z)$ for $z \in X(Z)$ span the space $A(Z)$ (whether we leave $Z$ fixed or take a union over various $Z$, such as all $Z$ with a fixed function field $K$, or even all $Z$ whose function field is cyclotomic). However, in the case $K=\Qb$, there is the following conjecture of Goncharov:

\begin{conj}[\cite{GonMPMTM}, Conjecture 7.4]\label{conj:gonch_exhaust}
The ring $A(\Qb)$ is spanned by elements of the form $\Liu_w(z)$ for $z$ a rational point or rational tangential basepoint of $\Poneminusthreepoints$ and $w$ a word in $e_0,e_1$.
\end{conj}

For our purposes, it is better to have a version of this conjecture with some control on ramification. The papers \cite{MTMUE} and \cite{MTMUEII} observed that one might need to replace $Z$ by an open subscheme $Z'$ and consider elements of the form $\Liu_w(z)$ for $z \in X(Z')$, even to generate $A(Z)$. More specifically, \cite[3.5]{MTMUE}, proposed an ``Archimedean condition,'' which is equivalent to the following definition:
\begin{defn}
We say that an integral scheme $Z$ with function field $\Qb$ is \emph{tapered} if there exists $N \in \Zb_{>0} \cup \{\infty\}$ such that $Z \cong \Spec{\Zb[1/S]}$ for $S=\{p \hspace{0.2cm} \mathrm{prime}; \, p \le N\}$.
\end{defn}

\cite[3.5]{MTMUE} then expressed the following as a ``Naive Hope,'' a phrase we retain:

\begin{naivehope}
If $Z$ is tapered, then the ring $A(Z)$ is spanned by elements of the form $\Liu_{w}(z)$ for $z$ a $Z$-integral point or tangential basepoint of $\Poneminusthreepoints$.
\end{naivehope}

This is known in the cases $N=1$ (\cite{BrownMTZ}) and $N=2$ (\cite{DelMuN}), and the case $N=\infty$ is Conjecture \ref{conj:gonch_exhaust}. Furthermore, \cite[3.6]{MTMUE} proves it in degrees $\le 2$. Finally, \cite[Conjecture 2.2.7]{MTMUEII} states a weaker version of this as a conjecture.

This history is important because the current paper is the first to put the preceding hopes and conjectures into practice. In our work, we do not need to make a precise conjecture; we will take all of this as expressing a principle that one may need to consider $z \in X(Z')$ for $Z'$ an open subscheme of $Z$. In Section \ref{sec:galois_coord_ex}, we will demonstrate this in practice by writing elements of $A(\Zb[1/3])$ in terms of values of polylogarithms at elements of $\Poneminusthreepoints(\Zb[1/6])$.

\begin{rem}\label{rem:gonch}
In doing so, we will find that all of these may be written in terms of single (rather than multiple) polylogarithms. This is in line with a conjecture of Goncharov about the depth filtration, as we now describe. We let
\[
\nf^G(Z) := \nf(Z)/[\nf(Z)_{<-1},[\nf(Z)_{<-1},\nf(Z)_{<-1}]],
\]
$\pi_1^{G}(Z)$ the corresponding quotient group, and $A^G(Z)$ its coordinate ring. It follows by the definition of $\pi_1^{\PL}(X)$ that
\[
\ZPLn = Z^1(\pi_1^{G}(Z)_{\ge n},\Pi_{\PL,n})^{\Gm}.
\]
In particular, each $\Liu_n(z)$ is in $A^G(Z)$, and Goncharov's conjecture (\cite[Conjecture 7.6]{GonMPMTM}) implies that $A^G(Z)$ is spanned by elements of the form $\Liu_n(z)$ when $Z=\Spec{\Qb}$.
Whenever $K$ is totally real, and $Z$ is missing at most two primes, $\nf^G(Z)$ agrees with $\nf(Z)$ in degrees $\ge -4$; hence it is reasonable to look only among single polylogarithms in the examples in this paper (but see Remark \ref{rem:gonch_ram}).
\end{rem}

\begin{rem}\label{rem:gonch_ram}
Formulating a precise conjecture that involves \emph{both} restriction to an open subscheme $Z' \subseteq Z$ \emph{and} restriction to the Goncharov quotient $\nf^G(Z)$ requires a lot of care; for we expect $\nf^G(Z)_n$ to grow with $n$, which by Siegel's Theorem implies that $A^G(Z)$ cannot be generated by single polylogarithms $\Liu_n(z)$ for $z \in X(Z)$, regardless of $Z$. Therefore, we may view Remark \ref{rem:gonch} as a heuristic for now, and our computations in \ref{sec:choosing_p_3_1/6} justify themselves. \cite[Conjecture 3.6]{PolGonII} gives a conjecture that takes both into account, but only in bounded weight.
\end{rem}

\subsection{Coordinates on the Space of Cocycles}\label{sec:coord_cocyc}

Fix an arbitrary family $\Si = \{\si_{n,i}\}$ of homogeneous free generators for $\nf(Z)$ with $\si_{n,i}$ in half-weight $-n$, and for each word $w$ of half-weight $-n$ in the above generators, let $f_w$ denote the associated element of $A_n = A(Z)_n$.

The following proposition was communicated to the authors in an unpublished letter by Francis Brown. It provides equivalent information to the geometric algorithm in \cite{MTMUEII}, but it allows one to avoid computing the logarithm in a unipotent group, making the algorithm much more practical.

\begin{defn}\label{defn:brown}
For an arbitrary cocycle $c \in {Z^{1, \Gm}_{\PL}}(R)$ for a $\Qb$-algebra $R$, a polylogarithmic word $\la$ of half-weight $-n$, and a word $w$ in the elements of $\Si$ of half-weight $-n$, let 
\[
\phi^w_\la(c) \in \QQ
\]
denote the associated matrix entry of $c^\sharp$, so that in the notation above, we have
\[
\Liu_\la(c) = \sum_{w} \phi_\la^w(c)f_w.
\]
\end{defn}

\begin{prop}\label{prop:brown}

Let $c \in {Z^{1, \Gm}_{\PL}}(R)$ for a $\Qb$-algebra $R$. For $0 \le r < n$, $\tau_1, \dots, \tau_r \in \Si_{-1}$, and $\si \in \Si_{r-n}$, we have
\[
\phi^{\si\tau_1\cdots\tau_r}_{\underbrace{ \footnotesize{e_1 e_0 \cdots e_0}}_{n}}(c)
=
\phi^{\tau_1}_{e_0}(c) \cdots \phi^{\tau_r}_{e_0} (c)
\phi^\si_{\underbrace{e_1 e_0\cdots e_0}_{n-r}}(c),
\] 
and all other matrix entries $\phi_\la^w(c)$ vanish.
\end{prop}

\begin{proof}
This amounts to a straightforward verification, but we nevertheless give the details. We begin with a formal calculation, in which $\Si_{-1}$ may be an arbitrary finite set, and $\{a^\tau\}_{\tau \in \Si_{-1}}$ a family of commuting coefficients. In this abstract setting, we claim that
\[
\left( \sum_{\tau \in \Si_{-1}} a^\tau f_\tau \right) ^{\Sha n}
=
n! \sum_{\tau_1, \dots, \tau_n \in \Si_{-1}}
 a^{\tau_1} \cdots a^{\tau_n} f_{\tau_1 \cdots \tau_n}.
\]
Indeed, the left side of the equation
\begin{align*}
	&= \sum_{\tau_1, \dots, \tau_n} (a^{\tau_1}f_{\tau_1}) \sha
		\cdots \Sha (a^{\tau_n}f_{\tau_n}) \\
	&= \sum_{\tau_1, \dots, \tau_n} a^{\tau_1}\cdots a^{\tau_n}
		\left( 
			\underset{\mbox{of }\tau_1, \dots, \tau_n}
				{\sum_{\mbox{permutations }p}}
			f_{\tau_1^p \cdots \tau_n^p}
		\right) \\
	&= \sum_p 
		\underbrace{
		\sum_{\tau_1, \dots, \tau_n} a^{\tau_1}
		\cdots a^{\tau_n} f_{\tau_1^p \cdots \tau_n^p}
			}_{\mbox{independent of }p}, \\
\end{align*}
which equals the right side of the equation.

Returning to our concrete situation, we apply the reduced coproduct $\Delta'$ to both sides of
\[
\sum_{|w| = -(n+1)} \phi^w_{n+1}(c) f_w 
= \Liu_{n+1}(c),
\]
and compute:
\begin{align}\label{align:brown}
\sum_{w', w''} \phi^{w'w''}_{n+1}(c) f_{w'}\otimes f_{w''}
 &= \sum_{i=1}^n  \Liu_{n+1-i}(c) \otimes \frac{(\logu(c))^{\Sha i}}{i!} \\
 &= \sum_{i=1}^n 
 	\left(
		\Liu_{n+1-i}(c)
		\otimes \frac
		{(\sum_{\tau \in \Si_{-1} } \phi^\tau_{e_0}(c) f_\tau)^{\Sha i}}
		{i!}
 	\right) \\
 &= \sum_{i=1}^n
  \sum_{\tau_1, \dots, \tau_i \in \Si_{-1}}
  \sum_{|w| = n+1-i}
  \phi^{\tau_1}_{e_0}(c) \cdots \phi^{\tau_i}_{e_0}(c) \phi^w_{n+1-i}(c)
   f_w \otimes f_{\tau_1 \cdots \tau_i}.\label{align:brown3}
\end{align}
Taking the coefficient of $f_v \otimes f_\tau$, with $\tau \in \Si_{-1}$, and $v$ an arbitrary word of length $n \ge 1$, we obtain
\[
\phi^{v \tau}_{n+1}(c) = \phi^\tau_{e_0}(c) \cdot \phi^v_n(c),
\]
while taking the coefficient of $f_v \otimes f_\si$, with $\si \in \Si_{-i \, < -1}$, and $v$ an arbitrary word of length $n+1-i \ge 1$, we obtain
\[
\phi^{v \si}_{n+1}(c) = 0. \qedhere
\] 
\end{proof}

We have a morphism
\[
\Cc_{\PL} \colon \ZPL \times \pi_1^\un(Z) \to \pi_1^\PL(X)
\]
given by
\[
(c,\gamma) \mapsto c(\gamma).
\]
We refer to $\Cc_{\PL}$ as the \emph{universal polylogarithmic cocycle}, and it is just the first component of the universal cocycle evaluation morphism $\evalu_\Pi$ for $\Pi = \pi_1^\PL(X)$. We define a morphism 
\[
\Psi: \ZPL \to 
\Spec \QQ \left[ \left\{ \Phi^\rho_\la \right\}
	_{wt(\rho) = wt(\la) } \right],
\]
where $\rho$ ranges over $\Si$ and $\la$ ranges over the set of polylogarithmic words, by
\[
c \mapsto (\phi^\rho_\la(c))_{\rho,\la}.
\]
We define a homomorphism of rings
\[
\theta^\sharp \colon 
\QQ[\logu, \Liu_1, \Liu_2, \dots]
\to
A(Z)  [ \left\{ \Phi^\rho_\la \right\}_{\rho, \la} ]
\]
by
\[
\logu
\mapsto
\sum_{\tau \in \Si_{-1}} f_\tau \Phi^\tau_{e_0} 
\]
and
\[
\Liu_n
\mapsto
\sum
_{
\substack{\tau_1, \dots, \tau_r \in \Si_{-1} \\
\si \in \Si_{-s} \\
r+s = n \\
1 \le s \le n
}
}
f_{\si \tau_1 \cdots \tau_r}
 \Phi^{\tau_1}_{e_0} \cdots \Phi^{\tau_r}_{e_0} 
\Phi^\si_{\underbrace{e_1 e_0 \cdots e_0}_{s}}
.
\]

\begin{cor}\label{cor:psi_isom}
The morphism $\Psi$ is an isomorphism. In particular, $\ZPL$ is canonically an affine space endowed with coordinates. Moreover, the triangle
\[
\xymatrix{
\ZPL \times \pi_1^\un(Z) \ar[rr]_-\sim^-{\Psi \times id}
\ar[dr]_-{\Cc_{\PL}}
&&
(\Spec \QQ [ \left\{ \Phi^\rho_\la \right\}_{\rho, \la} ] )
 \times \pi_1^\un(Z)
\ar[dl]^-{\theta}
\\
&
\pi_1^\PL(X) 
}
\]
commutes.
\end{cor}

\begin{proof}
The injectivity of $\Psi$, as well as the commutativity of the diagram, follow directly from Proposition \ref{prop:brown}. The surjectivity of $\Psi$ amounts to the statement that given any family $a^\rho_\la $ of elements of a $\QQ$-algebra $R$, the $R$-algebra homomorphism
\[
c^\sharp:   R[\logu, \Liu_1, \Liu_2, \dots] \to R \otimes A(Z)
\]
given by
\[
\logu \mapsto \sum_{\tau \in \Si_{-1}} a^\tau_{e_0} f_\tau
\]
and
\[
\Liu_n
\mapsto
\sum
_{
\substack{\tau_1, \dots, \tau_r \in \Si_{-1} \\
\si \in \Si_{r-n} \\
0 \le r \le n-1
}}
a^{\tau_1}_{e_0} \cdots a^{\tau_r}_{e_0} 
a^\si_{\underbrace{e_1 e_0 \cdots e_0}_{n-r}}
f_{\si \tau_1 \cdots \tau_r}
\]
is compatible with the coproduct. For $\logu$, this is because both sides are primitive. For $\Liu_n$, this follows by the computation of (\ref{align:brown})--(\ref{align:brown3}) and the fact that the terms for which $r=0$ are primitive.

\end{proof}

\subsubsection{Variant in Bounded Weight}

Let $n$ be a positive integer. Then there is a natural isomorphism $\ZPLn \xrightarrow{\Psi_n} (\Spec \QQ [ \left\{ \Phi^\rho_\la \right\}_{\rho, |\la| \le n} ] )$ such that the square
\[ 
\xymatrix{
\ZPL \ar[r]_-\sim^-{\Psi}
\ar[d]
&
\Spec \QQ [ \left\{ \Phi^\rho_\la \right\}_{\rho, \la} ] 
\ar[d]
\\
\ZPLn \ar[r]_-\sim^-{\Psi_n}
&
\Spec \QQ [ \left\{ \Phi^\rho_\la \right\}_{\rho, |\la| \le n} ] 
}
\]
commutes, where the vertical arrows are the natural projections.

\subsection{Kummer and Period Maps in Coordinates}\label{sec:kappa_in_coords}

Given $z \in X(Z)$, we recall that, by definition, 
\[
\logu(z) = \logu(\kappa(z)),
\]
\[
\Liu_n(z) = \Liu_n(\kappa(z)).
\]

We describe $\kappa_{\pf}$ in these coordinates. More precisely, for $z \in X(Z_{\pf})$, the value $\kappa_{\pf}$ is the element of $\Pi_{\PL,n}(\Qp)$ sending $\logu$ to $\logp(z)$ and $\Liu_n$ to $\Lip_n(z)$.

Finally, we describe $\mathrm{per}_{\pf}$ in these coordinates. More precisely, we have
\[
\mathrm{per}_{\pf}(\logu(z)) = \logp(z),
\]
and
\[
\mathrm{per}_{\pf}(\Liu_n(z)) = \Lip_n(z),
\]
which in particular expresses the commutativity of Kim's cutter.

\section{Computations for \texorpdfstring{$Z=\Spec{\Zb[1/S]}$}{Z=Spec(Z[1/S])}}\label{sec:ex_wt4}

For the rest of this article, we suppose that $Z$ has function field $\Qb$, so that $\pf \in Z$ is just the prime number $p$.

\subsection{Abstract Coordinates for \texorpdfstring{$Z=\Spec{\Zb[1/S]}$}{Z=Spec(Z[1/S])}}\label{sec:abs_over_S}

We let $\ell$ denote a prime number. We want to fix free generators $(\tau_\ell)_{\ell \in S}$ and $(\sigma_{2n+1})_{n \ge 1}$ for $\nf(Z)$, with $\tau_\ell$ in degree $-1$ and $\sigma_{2n+1}$ in degree $-2n-1$. We first recall some notation from \cite[\S 3.2]{MTMUEII}.

Letting $A_n = A(Z)_n$ denote the degree $n$ part of $A(Z)$, and $A_{>0} = \bigoplus_{n=1}^\infty A_n$, we let $E_n=E_n(Z)$ denote the kernel of $\Delta'_n$ and $D_n=D_n(Z)$ the subspace of decomposable elements of degree $n$, i.e., the degree $n$ elements in the image of
\[
A_{>0} \otimes A_{>0} \xrightarrow{\mathrm{mult}} A.
\]
We let $P_n=P_n(Z)$ denote a vector subspace of $A_n$ complementary to $E_n$ and $D_n$. The symbols $\Ec_n$, $\Dc_n$, and $\Pc_n$ refer to bases of $E_n$, $D_n$, and $P_n$, respectively.

\begin{prop}\label{prop:gens_A_over_ell}
One may choose free generators $(\tau_\ell)_{\ell \in S}$ and $(\sigma_{2n+1})_{n \ge 1}$ for $\nf(Z)$, with $\tau_\ell$ in degree $-1$ and $\sigma_{2n+1}$ in degree $-2n-1$, such that $f_{\tau_\ell} = \logu(\ell)$ and $f_{\sigma_{2n+1}} = \zetau(2n+1)$.

Furthermore a choice of $P_{2n+1}$ uniquely determines $\sigma_{2n+1}$.

\end{prop}

\begin{proof}
A computation using Corollary \ref{cor:motivic_polylog_coprod} shows that $\logu(\ell)$ and $\zetau(2n+1)$ are primitive elements of the Hopf algebra $A(Z)$ (or, in the terminology of \cite[1.6]{MTMUEII}, they lie in the space $E_n$ of extensions). In fact, by our knowledge of the rational algebraic K-theory of $Z$, we know that $E_n$ is one-dimensional when $n$ is odd and zero-dimensional otherwise. Therefore, the elements $\logu(\ell)$ and $\zetau(2n+1)$ must span the spaces $E_n$, and we take them as our $\Ec_n$.

By \cite[Proposition 3.2.3]{MTMUEII}, for a choice of $\Ec_n$, $\Dc_n$, and $\Pc_n$, we get a set of generators for the Lie algebra, which are dual to the elements of $\Ec_n$. In fact, the part of the condition of being dual that depends on $\Dc_n$ and $\Pc_n$ is that the element of the Lie algebra pairs to zero with all of $\Dc_n \cup \Pc_n$, so it in fact depends only on $P_n$ and $D_n$. The latter is uniquely determined, so a choice of $P_n$ determines such a choice of generators.
\end{proof}

\begin{rem}\label{rem:canonical2}
The choice of $\logu(\ell)$ and $\zetau(2n+1)$ corresponds to choosing generators for the rational algebraic K-groups of $Z$. The arbitrariness in choosing $P_n$ then corresponds precisely to the non-canonicity discussed toward the end of Section \ref{sec:mixed_tate_fund_grp}.
\end{rem}

%
%

Give such a set of generators, we get an abstract basis for $A(Z)$. For each word $w$ of half-weight $-n$ in the above generators, we have an element $w \in \Uc \nf(Z)$, and these form a basis of $\Uc \nf(Z)$. We let $(f_w)_w$ denote the dual basis for $A(Z)$. With the choices above, we have
\[
f_{\tau_\ell} = \logu(\ell)
\]
\[
f_{\sigma_{2n+1}} = \zetau(2n+1).
\]

In order to find bases of $P_n$ and verify relations between different $\Liu_n(z)$'s, we need to apply the reduced coproduct to reduce the computation in degree $n$ to the computation in degrees $m<n$. For this, we need the exact sequence
\begin{prop}\label{prop:exact_sequence}
The sequence

\[
0 \to E_n \to A_n \xrightarrow{\Delta'} \bigoplus
_{
\substack{
i+j = n\\
i,j \ge 1
}
} A_i \otimes A_j
\xrightarrow{\Delta' \otimes \mathrm{id} - \mathrm{id} \otimes \Delta'}
\bigoplus
_{
\substack{
i+j+k = n\\
i,j,k \ge 1
}
} A_i \otimes A_j \otimes A_k
\]
is exact, and $E_n = \Ext^1_{\MT(Z, \QQ)}(\Qb(0),\Qb(n)) = K_{2n-1}(Z)_\Qb$.
\end{prop}

\begin{proof}
We have the reduced cobar complex
\[
A_{>0} \xrightarrow{\Delta'} A_{>0} \otimes A_{>0} \xrightarrow{\Delta' \otimes \id - \id \otimes \Delta'} A_{>0} \otimes A_{>0} \otimes A_{>0} \xrightarrow{\Delta' \otimes \id \otimes \id - \id \otimes \Delta' \otimes \id + \id \otimes \id \otimes \Delta'} \cdots, \]
which is a complex of graded vector spaces.
\cite[\S 3.16]{BGSV90} notes that the $k$th cohomology of the degree $n$ piece of this complex is $\Ext^k_{\mathrm{GrComod}(A)}(\Qb(0),\Qb(n))$. The result then follows because the category of graded comodules over $A$ is the same as the category of representations of $\pi_1^{\MT}(Z)$, or equivalently, the category $\MT(Z, \QQ)$, and we know that
$
\Ext^1_{\MT(Z, \QQ)}(\Qb(0),\Qb(n)) = K_{2n-1}(Z)_\Qb$.

\end{proof}

We recall Definitions \ref{defn:delta_n} and \ref{defn:delta_ij}. In those definitions, we let $\Delta'_n$ denote the restriction of $\Delta'$ to $A_n$. For $i+j=n$, we let $\Delta_{i,j}'$ denote the component of $\Delta'_n$ landing in $A_i \otimes A_j$.

The following corollary will be useful for computations in half-weight $4$:

\begin{cor}\label{cor:injectivity_12}
If $K$ is totally real, then $\ker(\Delta_{1,2}') = \ker(\Delta'_3) = E_3$.
\end{cor}
\begin{proof}
Let $\alpha \in A_3$. Since $\Delta'_1$ is zero (by Remark \ref{rem:single_letter_prim}), we have
\begin{eqnarray*}
(\Delta'_2 \otimes \mathrm{id} - \mathrm{id} \otimes \Delta'_2)(\Delta'_3(\alpha)) &=& 
(\Delta'_2 \otimes \mathrm{id})(\Delta'_{2,1}(\alpha)) - (\mathrm{id} \otimes \Delta'_2)(\Delta'_{1,2}(\alpha))\\
&=& 0.
\end{eqnarray*}
As $K$ is totally real, it has no complex places, so we have $K_3(Z)_\Qb = 0$. Then by Proposition \ref{prop:exact_sequence}, we have that $E_2 = 0$. Therefore, $\Delta'_2$ is injective, hence $\Delta'_2 \otimes \id$ is injective on $A_2 \otimes A_1$. By the displayed equation and the previous sentence, it follows that if $\Delta'_{1,2}(\alpha)=0$, then $\Delta'_{2,1}(\alpha) = 0$ as well. Therefore, $\ker(\Delta_{1,2}') = \ker(\Delta'_3) = E_3$.
\end{proof}


\subsubsection{Coordinates on the space of cocycles for $Z=\Spec{\Zb[1/\ell]}$}\label{sec:coords_cocycles_one_over_ell}

Relative to the chosen coordinates, we name coordinates for $\ZPLn$ when $S=\{\ell\}$. Specifically, we set
\[
w_0 \colonequals \phi^{\tau_\ell}_{e_0}
\]
\[
w_1 \colonequals \phi^{\tau_\ell}_{e_1}
\]
\[
w_i \colonequals  \phi^{\sigma_{2i-1}}_{\underbrace{ \footnotesize{e_1 e_0 \cdots e_0}}_{2i-1}} \hspace{0.5cm} 2 \le i \le \left\lceil \frac{n}{2} \right\rceil
\]
in the notation of Proposition \ref{prop:brown}. In fact, for $i \ge 2$, $w_i$ makes sense even when $|S|>1$, and we use it in Section \ref{sec:case_1/3}.

If $z \in X(Z)$, we write $w_i(z)$ for $w_i(\kappa(z))$. In this notation, $w_0(z) = \ord_\ell z$, and $w_1(z) = - \ord_\ell(1-z)$, both by Fact \ref{fact:functional_equations}.

In this case, for $n \ge k$, we have
\[\evalu^\#_{\PL,n}(\logu) = \evalu^\#_{\PL}(\logu) = w_0 f_{\tau} \]
\begin{equation}\label{eqn:cc_li}
\evalu^\#_{\PL,n}(\Liu_k) = \evalu^\#_{\PL}(\Liu_k) = w_1 w_0^{k-1} f_{\tau}^k/k! + \sum_{i=2}^{\lceil \frac{k}{2} \rceil} w_i w_0^{k-2i+1} f_{\sigma_{2i-1} \tau^{k-2i+1}}.
\end{equation}


With this notation in hand, the reader may skip to Section \ref{sec:symmetrization} at this point if they so choose.

\subsection{The Geometric Step for \texorpdfstring{$\Zb[1/\ell]$}{Z[1/l]} in Half-Weight \texorpdfstring{$4$}{4}}\label{sec:geometric_step}

\subsubsection{Coordinates for the Galois Group}

Let $Z=\Zb[1/\ell]$. Then $\nf(Z)_{\ge -4}$ is three-dimensional as a vector space, generated by $\tau = \tau_\ell$, $\sigma = \sigma_3$, and $[\sigma,\tau]$. As $P_3=0$ when $|S|=1$, the elements $\tau,\sigma$ are already well-defined. We choose $f_\tau$, $f_\sigma$, and $f_{\sigma \tau}$ as a set of affine coordinates on $\pi_1^\un(Z)_{\ge -4}$.

\subsubsection{Coordinates for the Selmer Variety}

In this case, the only nonzero coordinates are $w_0$, $w_1$, and $w_2$.

\subsubsection{The Universal Cocycle Evaluation Morphism}

We now write the morphism $\evalu_{\PL,4}$ in these coordinates. We have $\evalu_{\PL,4}^\#(f_i)=f_i$ for $i=\tau,\sigma,\sigma\tau$. 

Using (\ref{eqn:cc_li}), we find that $\evalu_{\PL,4}^\#(\logu)=w_0 f_\tau$, $\evalu_{\PL,4}^\#(\Liu_1)=w_1 f_\tau$, and $$\evalu_{\PL,4}^\#(\Liu_2) = w_0 w_1 f_{\tau}^2/2.$$

At this point, we already see that the function on $\Pi_{\PL,4} \times \pi_1^\un(Z)_{\ge -4}$ given by $$\Liu_2 - \frac{1}{2} \logu \Liu_1$$ vanishes on the image of $\evalu_{\PL,4}$, i.e., is in $\Ic^Z_{\PL,4}$.

Again, by Equation (\ref{eqn:cc_li}), we have
\[\evalu_{\PL,4}^\#(\Liu_3) = w_1 w_0^2 f_\tau^3/6 + w_2 f_\sigma,\]
\[\evalu_{\PL,4}^\#(\Liu_4) = w_1 w_0^3 f_\tau^4/24 + w_0 w_2 f_{\sigma \tau}.\]

To construct a second element of $\Ic^Z_{\PL,4}$, we first eliminate $w_2$ by considering
\begin{eqnarray*}
\evalu_{\PL,4}^\#(f_\sigma f_\tau \Liu_4 - f_{\sigma \tau} \logu \Liu_3) &=& w_1 w_0^3 f_\sigma f_\tau^5/24 -   f_{\sigma\tau} w_1 w_0^3 f_\tau^4/6\\ 
&=& \frac{w_1 w_0^3 f_\tau^4}{24} \left(f_\sigma f_\tau - 4 f_{\sigma \tau}\right)\\
&=& \frac{\evalu_{\PL,4}^\#((\logu)^3 \Liu_1)}{24} \left(f_\sigma f_\tau - 4 f_{\sigma \tau}\right)
\end{eqnarray*}

It follows that
\[
\evalu_{\PL,4}^\#\left(f_\sigma f_\tau \Liu_4 - f_{\sigma \tau} \logu \Liu_3 - \frac{(\logu)^3 \Liu_1}{24} \left(f_\sigma f_\tau - 4 f_{\sigma \tau}\right)\right) = 0
\]

In other words, 

\begin{prop}\label{prop:geom_wt4}
The two functions
\[\Liu_2 - \frac{1}{2} \logu \Liu_1\]
\[f_\sigma f_\tau \Liu_4 - f_{\sigma \tau} \logu \Liu_3 - \frac{(\logu)^3 \Liu_1}{24} \left(f_\sigma f_\tau - 4 f_{\sigma \tau}\right)\]
on $\Pi_{\PL,4} \times \pi_1^\un(Z)_{\ge -4}$ are in $\Ic^Z_{\PL,4}$ for $Z = \Spec{\Zb[1/\ell]}$.
\end{prop}

\subsection{Coordinates on the Galois Group for \texorpdfstring{$Z=\Zb[1/\ell]$}{Z=Z[1/l]}}\label{sec:galois_coord_ex}

In order to evaluate the functions of Proposition \ref{prop:geom_wt4}, we need to choose a prime $p \neq \ell$ and interpret $f_\tau$, $f_\sigma$, and $f_{\sigma\tau}$ in such a way that we can take (approximate) their $p$-adic periods. Essentially, this means writing them as special values of polylogarithms. As mentioned in Proposition \ref{prop:gens_A_over_ell}, we have chosen the first two to correspond to $\logu(\ell)$ and $\zetau(3)$, respectively. It remains to understand $f_{\sigma \tau}$.

This will depend on $\ell$. We start with the case $\ell=2$, in an effort to re-derive \cite[Theorem 1.16]{MTMUE}.

\subsubsection{The Case $Z=\Zb[1/2]$}

We let $A=A(Z)$ as usual. We compute using Corollary \ref{cor:motivic_polylog_coprod} and Fact \ref{fact:functional_equations} that $$\Delta'_{1,2}(\Liu_3(1/2)) = \Liu_1(1/2) \otimes (\logu(1/2))^2/2 = -\logu(1/2) \otimes (\logu(2))^2/2 = f_\tau \otimes f_\tau^2/2.$$
Using the fact that $\Delta$ is a ring homomorphism, we note that $\Delta'_{1,2}(\logu(2)^3)=3 \logu(2) \otimes \logu(2)^2$. Therefore, Corollary \ref{cor:injectivity_12} implies that $$\Liu_3(1/2) - \frac{(\logu(2))^3}{6}$$ is in $E_3$. As $E_3 = K_5(Z)_{\Qb}$ is one-dimensional, this is a rational multiple of $\zetau(3)$.

The identity in the appendix to \cite{MTMUE} says that
\begin{equation}\label{eqn:liu_3_half}\Liu_3(1/2) = \frac{(\logu(2))^3}{6} + \frac{7}{8} \zetau(3).\end{equation}

By Definition \ref{defn:brown}, we have
$$\evalu_{\PL,4}^\#(\Liu_4) = w_1 w_0^3 f_\tau^4/24 + w_0 w_2 f_{\sigma \tau}.$$

It is easy to check using the formulas at the end of Section \ref{sec:coords_cocycles_one_over_ell} that $w_0(\kappa(1/2)) = -1$ and $w_1(\kappa(1/2))=1$, and (\ref{eqn:liu_3_half}) together with Definition \ref{defn:brown} implies that $w_2(\kappa(1/2)) = 7/8$.

We therefore get $$\Liu_4(1/2) = - (\logu(2))^4/24 - \frac{7}{8} f_{\sigma \tau}.$$ It follows that $$f_{\sigma \tau} = -\frac{8}{7}\left(\frac{\logu(2)^4}{24}+\Liu_4(1/2)\right).$$

\subsubsection{Choosing $P_3(\Spec{\Zb[1/6]})$}\label{sec:choosing_p_3_1/6}

To deal with the case $Z=\Spec{\Zb[1/3]}$, we have to consider $Z'=\Spec{\Zb[1/6]}$, since $X(Z)$ is empty. In this case, the Lie algebra $\nf(Z')$ is generated by $\tau_2, \tau=\tau_3$, and $\sigma = \sigma_3$. 

In this case, $P_1(Z')$ is zero. As $P_3(Z')$ is nonzero, there is some choice in the definition of $\sigma$ as an element of $\nf(Z')$. We therefore seek to choose a set $\Pc_3(Z')$.

Before that, we deal with half-weights $1$ and $2$. $A(Z')_1$ has basis $\{\logu(2),\logu(3)\}$, and we have the following lemma for $A(Z')_2$:

\begin{lemma}\label{lemma:basis_A(Z')_2}
The set $\{\logu(2)^2,  \logu(2) \logu(3),  \logu(3)^2, \Liu_2(-2)\}$ is a basis of $A(Z')_2$.
\end{lemma}

\begin{proof}
As $E_2(Z')=0$, the map $\Delta'_2$ is injective. Then $D_2(Z')$ is determined, and a natural basis for it is $\Dc_2(Z') = \{\logu(2)^2, \logu(2) \logu(3), \logu(3)^2\}$. Then 
\[\Delta'(\Liu_2(-2)) = \Liu_1(-2) \otimes \logu(-2) = - \logu(3) \otimes \logu(2),
\]
which is independent from $$\{\Delta' \logu(2)^2, \Delta' \logu(2) \logu(3), \Delta' \logu(3)^2\}$$ in $A(Z')_1 \otimes A(Z')_1$ (as seen by checking via the basis of $A(Z')_1 \otimes A(Z')_1$ induced by the basis $\{\logu(2),\logu(3)\}$ of $A(Z')_1$). 
\end{proof}

In fact, $\Liu_2(-2) = - f_{\tau \tau_2}$, since they have the same reduced coproduct.

\begin{prop}\label{lemma:basis_A(Z')_3}
We may take $\{\Liu_3(-2), \Liu_3(3)\}$ as $\Pc_3(Z')$, and therefore
\begin{align*}\{\logu(2)^3, \logu(3)^3, \logu(2)^2 \logu(3), \logu(2) \logu(3)^2,\\ \logu(2) \Liu_2(-2), \logu(3) \Liu_2(-2), \Liu_3(-2), \Liu_3(3),\zetau(3)\}
\end{align*}
is a basis of $A(Z')_3$.
\end{prop}

\begin{proof}
By taking all degree 3 products of basis elements of $A(Z')_1$ and $A(Z')_2$, we may determine a basis for the decomposables in $A(Z')_3$:
\[\Dc_3(Z') = \{\logu(2)^3, \logu(3)^3, \logu(2)^2 \logu(3), \logu(2) \logu(3)^2, \logu(2) \Liu_2(-2), \logu(3) \Liu_2(-2)\}.
\]

We would like to show that \begin{align*}\Bc \colonequals \{\logu(2)^3, \logu(3)^3, \logu(2)^2 \logu(3), \logu(2) \logu(3)^2,\\ \logu(2) \Liu_2(-2), \logu(3) \Liu_2(-2), \Liu_3(-2), \Liu_3(3)\}\end{align*} is a basis of $A(Z')_3/E_3(Z')$, which would imply that $\{\Liu_3(-2), \Liu_3(3)\}$ can be taken as $\Pc_3(Z')$. To do this, we apply $\Delta'_{1,2}$ to each element of $\Bc$ and expand in the basis
\begin{align*}\{\logu(2) \otimes \logu(2)^2, \logu(2) \otimes \logu(2) \logu(3), \logu(2) \otimes \logu(3)^2, \logu(2) \otimes \Li_2(-2),\\ \logu(3) \otimes \logu(2)^2, \logu(3) \otimes \logu(2) \logu(3), \logu(3) \otimes \logu(3)^2, \logu(3) \otimes \Li_2(-2)\}
\end{align*}
of $A(Z')_1 \otimes A(Z')_2$. Using Corollary \ref{cor:motivic_polylog_coprod}, this produces the matrix
\[
\begin{array}{|c|c|c|c|c|c|c|c|c|}
\hline
\logu(2) \otimes \logu(2)^2  & 3 & 0 & 0 & 0 & 0 & 0 & 0 & 0\\ \hline
\logu(2) \otimes \logu(2) \logu(3) & 0 & 0 & 2 & 0 & 0 & 0 & 0 & 0\\ \hline
\logu(2) \otimes \logu(3)^2 & 0 & 0 & 0 & 1 & 0 & 0 & 0 & -1/2\\ \hline
\logu(2) \otimes \Li_2(-2) & 0 & 0 & 0 & 0 & 1 & 0 & 0 & 0\\ \hline
\logu(3) \otimes \logu(2)^2 & 0 & 0 & 1 & 0 & -1 & 0 & -1/2 & 0\\ \hline
\logu(3) \otimes \logu(2) \logu(3) & 0 & 0 & 0 & 2 & 0 & -1 & 0 & 0\\ \hline
\logu(3) \otimes \logu(3)^2 & 0 & 3 & 0 & 0 & 0 & 0 & 0 & 0\\ \hline
\logu(3) \otimes \Li_2(-2) & 0 & 0 & 0 & 0 & 0 & 1 & 0 & 0\\ \hline
\end{array},
\]
where the columns correspond to the elements of $\Bc$. This matrix has determinant $9$, which by Corollary \ref{cor:injectivity_12} and the fact that $A(Z')_3/E_3(Z')$ is eight-dimensional (by our knowledge of the shuffle algebra $A(Z')$) implies that $\Bc$ is in fact a basis of $A(Z')_3/E_3(Z')$. As $\zetau(3)$ generates $E_3(Z')$,
\[
\Bc \cup \{\zetau(3)\}
\]
is a basis of $A(Z')_3$, as desired.
\end{proof}

From now on, we choose $P_3(Z')$ to be the space generated $\Liu_3(-2)$ and $\Liu_3(3)$, and $\sigma$ such that it pairs to $0$ with this choice of $P_3(Z')$.

\subsubsection{The Case $Z=\Zb[1/3]$}\label{sec:case_1/3}

Armed with our choice of $\sigma \in \nf(Z')$, we seek to write $f_{\sigma \tau} \in A(Z)$ as an explicit combination of motivic polylogarithms. With our choice of $\sigma$, we have $w_2(-2)=w_2(3)=0$ because $\Liu_3(-2)$ and $\Liu_3(-3)$ are elements of $\Pc_3$.

From the matrix above, we see that
\[
\Delta'_{1,2} \Liu_3(3) = - \frac{\logu(2) \otimes \logu(3)^2}{2}.
\]
We also compute
\[
\Delta'_{1,2} \Liu_3(9) = \Liu_1(9) \otimes \frac{\logu(9)^2}{2} = - 6 \logu(2) \otimes \logu(3)^2,
\]
so $\Liu_3(9) = 12 \Liu_3(3)$ is in $E_3(Z')$, hence a rational multiple of $\zetau(3)$ (because $E_3(Z') = K_5(Z')_{\Qb}$ is one-dimensional). In fact, since $w_2(3)=0$, this rational number is the $f_{\sigma}$-coordinate of $\Liu_3(9)$ in the basis $(f_w)_w$ of $A(Z')$, hence it equals $w_2(9)$ because $f_\sigma = \zetau(3)$.

Numerical computation using the code \cite{ChatDC} in the $5$-adic and $7$-adic realizations suggests that $$w_2(9) = \frac{\Liu_3(9)-12\Liu_3(3)}{\zetau(3)} = -\frac{26}{3}.$$ This is proven in the Appendix (Section \ref{sec:appendix}).

We may now compute using Definition \ref{defn:brown} that
\[
\Liu_4(3) = w_2(3) \phi^{\tau}_{e_0}(3) f_{\sigma \tau} + \phi^{\tau_2}_{e_1}(3) \phi^{\tau}_{e_0}(3)^3 f_{\tau_2 \tau \tau \tau}\\
= -f_{\tau_2 \tau \tau \tau}
\]
and
\[
\Liu_4(9) = w_2(9) \phi^{\tau}_{e_0}(9) f_{\sigma \tau} + \phi^{\tau_2}_{e_1}(9) \phi^{\tau}_{e_0}(9)^3 f_{\tau_2 \tau \tau \tau}\\
= 2 w_2(9) f_{\sigma \tau} - 24 f_{\tau_2 \tau \tau \tau} 
\]




This implies that
$$
-\frac{12}{w_2(9)} \Liu_4(3) + (2 w_2(9))^{-1} \Liu_4(9) =
\frac{18}{13} \Liu_4(3) - \frac{3}{52} \Liu_4(9)
=
f_{\sigma \tau}.
$$

We note that this corresponds precisely to $f_{\sigma \tau}$ in the image of $\Oc(\pi_1^\un(Z)) \hookrightarrow \Oc(\pi_1^\un(Z'))$.

Plugging this into the second equation of Proposition \ref{prop:geom_wt4}, we find
\begin{thm}\label{thm:computation_wt4}
The element
\begin{align*}
\zetau(3) \logu(3) \Liu_4 - \left(\frac{18}{13} \Liu_4(3) - \frac{3}{52} \Liu_4(9)\right) \logu \Liu_3 \\
- \frac{(\logu)^3 \Liu_1}{24} \left(\zetau(3) \logu(3) - 4 \left(\frac{18}{13} \Liu_4(3) - \frac{3}{52} \Liu_4(9)\right)\right)
\end{align*}
of $\Oc(\Pi_{\PL,4} \times \pi_1^\un(Z))$ is in $\Ic^Z_{\PL,4}$ for $Z=\Spec{\Zb[1/3]}$.
\end{thm}

For $p \neq 2,3$, the corresponding Coleman function is

\begin{align}\label{eqn:the_function_p}
\zetap(3) \logp(3) \Lip_4(z) - \left(\frac{18}{13} \Lip_4(3) - \frac{3}{52} \Lip_4(9)\right) \logp(z) \Lip_3(z) \\
- \frac{(\logp(z))^3 \Lip_1(z)}{24} \left(\zetap(3) \logp(3) - 4 \left(\frac{18}{13} \Lip_4(3) - \frac{3}{52} \Lip_4(9)\right)\right). \nonumber
\end{align}

\subsection{The Chabauty-Kim Locus for \texorpdfstring{$\Zb[1/3]$}{Z[1/3]} in Half-Weight \texorpdfstring{$4$}{4}}\label{sec:CK_locus_1/3}

As noted in \cite[\S 8.2]{nabsd}, the function $\Lip_2(z)-\frac{1}{2} \logp(z) \logp(1-z)$ has the zero set $\{2,\frac{1}{2},-1\}$ for $p=3,5,7$. By numerical evaluation of (\ref{eqn:the_function_p}) at $2$, $\frac{1}{2}$, and $-1$ using \cite{ChatDC}, we conclude that $2, \frac{1}{2} \notin X(Z_{\pff})_{\PL,4}$. It follows that:

\begin{thm}\label{thm:computation_loci}
For $Z = \Spec{\Zb[1/3]}$, we have $X(Z_{\pff})_{\PL,4} \subseteq \{-1\}$ for $p=5,7$.
\end{thm}

\section{Insufficiency of Polylogarithmic Quotient and \texorpdfstring{$S_3$}{S3}-Symmetrization}\label{sec:symmetrization}

We first recall some notation from Section \ref{sec:coord_cocyc}. For a family $\Sigma$ of generators of the Lie algebra $\nf(Z)$, we have a corresponding shuffle (vector space) basis of $A(Z)$ denoted by $f_w$ for $w$ a word in $\Sigma$. As well, $\Oc(\pi_1^\un(X))$ has a basis of elements $\Liu_{\lambda}$ corresponding to words $\lambda$ in $\{e_0,e_1\}$. We defined regular functions $\phi_{\lambda}^w$ on the space $\ZPLn$ of cocycles, for $w$ a word in $\Sigma$ and $\lambda$ a word in $\{e_0,e_1\}$ of length at most $n$, such that for $c \in \ZPLn(R)$, we have
\[
\Liu_{\lambda}(c) = \sum_{w} \phi_{\lambda}^w(c) f_w.
\]

We also recall some notation from Section \ref{sec:abs_over_S}. Letting $Z=\Zb[1/\ell]$, we fix generators $(\tau=\tau_{\ell})$ and $(\sigma_{2n+1})_{n \ge 1}$ for $\nf(Z)$, such that $f_{\tau}=\logu(\ell)$, and $f_{\sigma_{2n+1}} = \zetau(2n+1)$. Finally, we set
\[
w_0 \colonequals \phi^{\tau_\ell}_{e_0}
\]
\[
w_1 \colonequals \phi^{\tau_\ell}_{e_1}
\]
\[
w_i \colonequals  \phi^{\sigma_{2i-1}}_{\underbrace{ \footnotesize{e_1 e_0 \cdots e_0}}_{2i-1}} \hspace{0.5cm} 2 \le i \le \left\lceil \frac{n}{2} \right\rceil.
\]

This section does not depend on the material in Sections \ref{sec:geometric_step}-\ref{sec:CK_locus_1/3}.

\subsection{Proof that \texorpdfstring{$-1 \in X(Z_{\pff})_{\PL,n}$}{-1 is in X(Zp)(PL,n)} for \texorpdfstring{$Z=\Spec{\Zb[1/\ell]}$}{Z=Spec(Z[1/l])}}

We fix a set of generators as in Section \ref{sec:abs_over_S}.

We first need the following lemma:

\begin{lemma}\label{lemma:lip_identity}
We have $\Lip_k(-1)=0$ for $k \ge 2$ even.
\end{lemma}
\begin{proof}
This follows from the identity $2^{-k+1} \Lip_k(z^2) = \Lip_k(z)+\Lip_k(-z)$, which is \cite[Proposition 6.1]{Coleman} for $m=2$. Indeed, setting $z=1$ in the identity shows $\Lip_k(-1) = (2^{-k}-1) \Lip_k(1)$, and since $\Lip_k(1) = \zetap(k)$, which is $0$ for $k$ even, we have $\Lip_k(-1)=0$.
\end{proof}

\begin{thm}\label{thm:counterexample}
For any prime $\ell$ and positive integer $n$, we have
\[
-1 \in X(Z_{\pff})_{\PL,n},
\]
where $Z = \Spec{\Zb[1/\ell]}$.
\end{thm}

\begin{proof}
We use the coordinates of Section \ref{sec:abs_over_S}, and we let $\Kc$ and $\Izn$ be as defined in Section \ref{sec:the_geometric_step}. We also write $\tau = \tau_\ell$. We have $f_\tau = \logu(\ell)$.

To prove the theorem, we produce an element $c_{-1}$ of $\ZPLn(\Kc)$ whose image $\alpha_{-1}$ under $\evalu_{\PL,n}$ lies in $\Pi_{\PL,n}(A(Z)) \subseteq \Pi_{\PL,n}(\Kc)$ and satisfies 
\[\mathrm{per}_{\pff}(\alpha_{-1}) = \kappa_{\pff}(-1) \in \Pi_{\PL,n}(\Qp).\]
Since any element $f$ of $\Oc(\Pi_{\PL,n} \times \pi_1^\un(Z)) \cap \Izn$ vanishes on the image of $\evalu^{\Kc}_{\PL,n}$, we have that $f$ vanishes on $\alpha_{-1}$ and therefore on $\kappa_{\pff}(-1)$, which proves the theorem.

We define $c_{-1}$ by setting
\[w_0(c_{-1}) \colonequals 0,\]
\[w_1(c_{-1})  \colonequals \frac{\Liu_1(-1)}{\logu(\ell)},\]
\[w_i(c_{-1}) \colonequals \frac{\Liu_{2i-1}(-1)}{\zetau(2i-1)} \hspace{0.5cm} i = 2, \cdots, \left\lceil \frac{n}{2} \right\rceil.\]

We let $\alpha_{-1} \colonequals \evalu_{\PL,n}(c_{-1})$. To show that
\[\mathrm{per}_{\pff}(\alpha_{-1}) = \kappa_{\pff}(-1) \in \Pi_{\PL,n}(\Qp),\]
we note that $\Oc(\Pi_{\PL,n})$ is generated as an algebra by $\{\logu,\Liu_1,\cdots,\Liu_n\}$. Therefore, it suffices to show that
\[
\mathrm{per}_{\pff}(\logu(\alpha_{-1})) = \mathrm{per}_{\pff}(\logu(-1)),
\]
and that
\[
\mathrm{per}_{\pff}(\Liu_k(\alpha_{-1})) = \mathrm{per}_{\pff}(\Liu_k(-1))
\]
for $1 \le k \le n$. 

We have $\evalu_{\PL,n}^{\#}(\Liu_1)= w_1 f_\tau$, so $$\Liu_1(\alpha_{-1}) = w_1(c_{-1}) f_\tau = \frac{\Liu_1(-1)}{\logu(\ell)} \logu(\ell) = \Liu_1(-1).$$

For $k \ge 3$ odd, we have by (\ref{eqn:cc_li}) that
\begin{eqnarray*}\evalu^\#_{\PL,n}(\Liu_k) &=& w_1 w_0^{k-1} f_{\tau}^k/k! + \sum_{i=2}^{\frac{k+1}{2}} w_i w_0^{k-2i+1} f_{\sigma_{2i-1} \tau^{k-2i+1}}\\
&=& w_0\left(w_1 w_0^{k-1} f_{\tau}^k/k! + \sum_{i=2}^{\frac{k-1}{2}} w_i w_0^{k-2i} f_{\sigma_{2i-1} \tau^{k-2i+1}}\right) + w_{\frac{k+1}{2}} f_{\sigma_k},
\end{eqnarray*}
so by $w_0(c_{-1})=0$, we have
\begin{eqnarray*}
\Liu_k(\alpha_{-1}) &=& w_{\frac{k+1}{2}}(c_{-1}) f_{\sigma_k}\\
&=& \frac{\Liu_{k}(-1)}{\zetau(k)} \zetau(k)\\
&=& \Liu_k(-1).
\end{eqnarray*}

Thus for $k$ odd, we have $\Liu_k(\alpha_{-1})=\Liu_k(-1)$, so
\[
\mathrm{per}_{\pff}(\Liu_k(\alpha_{-1})) = \mathrm{per}_{\pff}(\Liu_k(-1))
\]
for $k$ odd, as desired.

For $k \ge 2$ even, we have by (\ref{eqn:cc_li}) that
\begin{align*}
\evalu^\#_{\PL,n}(\Liu_k) = w_1 w_0^{k-1} f_{\tau}^k/k! + \sum_{i=2}^{ \frac{k}{2}} w_i w_0^{k-2i+1} f_{\sigma_{2i-1} \tau^{k-2i+1}}\\ = w_0\left(w_1 w_0^{k-1} f_{\tau}^k/k! + \sum_{i=2}^{\frac{k}{2}} w_i w_0^{k-2i} f_{\sigma_{2i-1} \tau^{k-2i+1}}\right),
\end{align*}
and since $w_0(c_{-1})=0$, we have $\Liu_k(\alpha_{-1}) = 0$. We also have $\evalu_{\PL,n}^{\#}(\logu) = w_0 f_\tau$, so $$\logu(\alpha_{-1}) = w_0(c_{-1}) f_{\tau} = 0.$$

To show that $\mathrm{per}_{\pff}(\alpha_{-1}) = \kappa_{\pff}(-1)$, we are hence reduced to showing that $\mathrm{per}_{\pff}(\Liu_k(-1))=0$ for $k$ even, and that $\mathrm{per}_{\pff}(\logu(-1))=0$. By Lemma \ref{lemma:lip_identity}, $\Lip_k(-1)=0$ for $k \ge 2$ even, and since $\mathrm{per}_{\pff}(\Liu_k(-1))=\Lip_k(-1)$, this implies that $\mathrm{per}_{\pff}(\Liu_k(-1))=0$ for $k \ge 2$ even. We also have $\mathrm{per}_{\pff}(\logu(-1)) = \logp(-1)=0$. 

Finally, recall that we needed to show that $\alpha_{-1} \in \Pi_{\PL,n}(A(Z)) \subseteq \Pi_{\PL,n}(\Kc)$. This follows because we have shown that $\logu(\alpha_{-1})=0 \in A(Z)$, $\Liu_k(\alpha_{-1})=0 \in A(Z)$ for $k$ even, and $\Liu_k(\alpha_{-1}) = \Liu_k(-1) \in A(Z)$ for $k$ odd.


\end{proof}

\begin{rem}\label{rem:regular_p}
Recall that by Remark \ref{rem:kimvsus}, the $p$-adic period conjecture (Conjecture \ref{conj:period_conj}) implies that our locus is the same as Kim's locus $X(Z_{\pff})_{n,\Kim}^Z$. Similarly, one may define a polylogarithmic version $X(Z_{\pff})_{n,\PL,\Kim}^Z$ of Kim's locus, which is contained in our locus $X(Z_{\pff})_{n,\PL}$ and is equal if the $p$-adic period conjecture is true.

In fact, one could modify the proof of Theorem \ref{thm:counterexample} to directly prove that $-1 \in X(Z_{\pff})_{n,\PL,\Kim}^Z$ as long as $\zetap(2i-1) \neq 0$ for $2 \le i \le \left\lceil \frac{n}{2} \right\rceil$. This latter fact is known whenever $p$ is a regular prime (c.f. \cite[Remark 2.20(i)]{FurushoI}).
\end{rem}

\subsection{\texorpdfstring{$S_3$}{S3}-Symmetrization}


We recall our strengthening of Conjecture \ref{conj:kim}. The $S_3$-action on $X$ allows us to define $X(Z_{\pff})_{\PL,n}^{S_3} \colonequals \bigcap_{\sigma \in S_3} \sigma(X(Z_{\pff})_{\PL,n})$. Our symmetrized conjecture is that
\begin{conj}\label{conj:polylog_kim_symm2}
$X(Z) = X(Z_{\pff})_{\PL,n}^{S_3}$ for sufficiently large $n$.
\end{conj}
As shown in Proposition \ref{prop:different_conjs}, this conjecture implies Conjecture \ref{conj:kim} and hence Conjecture \ref{conj:kim1}, the latter of which appeared in \cite{nabsd}.

\subsubsection{Verification for $Z=\Spec{\Zb[1/3]}$}

We can use our computations in Section \ref{sec:ex_wt4} to verify a case of this conjecture:

\begin{thm}\label{thm:verification_computation}
For $Z=\Spec{\Zb[1/3]}$ and $p=5,7$, Conjecture \ref{conj:polylog_kim_symm2} (and hence Conjecture \ref{conj:kim}) holds (with $n=4$).
\end{thm}

\begin{proof}
By Theorem \ref{thm:computation_loci}, $X(Z_{\pff})_{\PL,4} \subseteq \{-1\}$. But $-1$ is not fixed by the action of $S_3$, so 
$$
X(Z_{\pff})_{\PL,4}^{S_3} \colonequals \bigcap_{\sigma \in S_3} \sigma(X(Z_{\pff})_{\PL,4}) = 
\emptyset$$
for $p=5,7$. In particular, Conjecture \ref{conj:polylog_kim_symm2} and hence Conjecture \ref{conj:kim} holds in these cases.

\end{proof}

\section{Appendix: Proof that \texorpdfstring{$w_2(9)=-\frac{26}{3}$}{w2(9)=-26/3}}\label{sec:appendix}


We use the Kummer-Spence relation for the trilogarithm to show that $w_2(9)=-\frac{26}{3}$, or equivalently, that
\[
\Liu_3(9)-12 \Liu_3(3) = -\frac{26}{3} \zetau(3)
\]

The reasoning is a bit indirect, involving both $p$-adic and complex polylogarithms. Ideally, we would use a motivic version of the Kummer-Spence relation, which should be possible to prove.
However, this does not exist in the literature (nor does a $p$-adic version), and the most direct route seemed to be to cite the Kummer-Spence relation in the single-valued complex case and proceed from there. Furthermore, if we had a `single-valued' Archimedean period map $A(Z) \to \Cb$ sending $\Liu_k(z)$ to $P_k(z)$ (analogous to that in \cite{BrownSingle} for $Z=\Spec{\Zb}$), we could directly prove the desired identity using Lemma \ref{lemma:complex_ident}. However, lacking such a period map, we found it most convenient to go via a $p$-adic identity and then use the map $\mathrm{per}_{p}$.

Applying $\mathrm{per}_p$, it suffices to prove the corresponding identity
\[
\Lip_3(9)-12 \Lip_3(3) = -\frac{26}{3} \zetap(3)
\]
for any prime $p>3$ for which $\zetap(3) \neq 0$. This is known for any regular prime, so we may take, for example, $p=5$.

By \cite[Proposition 6.1]{Coleman} for $k=3$ and $m=2$, we have
\[
\Lip_3(9) = 4 \Lip_3(3) + 4 \Lip_3(-3),
\]
so one may easily check that it is equivalent to prove
\[
\Lip_3(-3)-2\Lip_3(3) = -\frac{13}{6} \zetap(3).
\]

\begin{rem}\label{rem:ident_already}
We note that this identity was already mentioned on p.6 of \cite{GanglZagier} in the form ``$\mathcal{L}_3(\xi_2)=\frac{13}{6}\zeta(3)$,'' although no proof was given (and according to one of the authors, it may have been verified only numerically).
\end{rem}

\begin{defn}
Following \cite{ZagierAppendix}, we set
\[
P_3(z) = \mathrm{Re}\left(\Li_3(z)-\Li_2(z)\log|z|+\frac{1}{3} \log|1-z|\log^2|z|\right).
\]
\end{defn}

\begin{lemma}\label{lemma:complex_ident}
We have
\[
-\frac{13}{6} \zeta(3) = P_3(-1/3)-2P_3(1/3).\]
\end{lemma}

\begin{proof}
By \cite{ZagierAppendix}, Section 7, Example 3, we have the ``Kummer-Spence'' functional equation
\begin{align*}
    2P_3(x) + 2P_3(y) + 2P_3\left(\frac{x(1-y)}{x-1}\right) + 2P_3\left(\frac{y(1-x)}{y-1}\right) + 2P_3\left(\frac{1-x}{1-y}\right)\\ + 2P_3\left(\frac{x(1-y)}{y(1-x)}\right)
    - P_3\left(xy\right) - P_3\left(\frac{x}{y}\right) - P_3\left(\frac{x(1-y)^2}{y(1-x)^2}\right) = 2 \zeta(3).
\end{align*}

Plugging in $x=-1$ and $y=1/3$, we get
\begin{eqnarray*}
2 \zeta(3) &=& 2P_3(-1)+2P_3(1/3)+2P_3(1/3)+2P_3(-1)+ 2P_3(3)+\\
&& 2P_3(-1)-P_3(-1/3)-P_3(-3)-P_3(-1/3)\\
&=& 6P_3(-1)+4P_3(1/3)+2P_3(3)-P_3(-3)-2P_3(-1/3)
\end{eqnarray*}

Applying the inversion relation $P_3(x)=P_3(1/x)$, we get
\[
2 \zeta (3) = 6P_3(-1)+6P_3(3)-3P_3(-3).
\]

From the dispersion relation, we get $P_3(1)=P_3((-1)^2) = 4P_3(1)+4P_3(-1)$, so $P_3(-1) = -\frac{3}{4}\zeta(3)$, hence $6P_3(-1) = -\frac{9}{2} \zeta(3)$. Therefore, we get
\[
2 \zeta (3) = -\frac{9}{2}\zeta(3)+6P_3(1/3)-3P_3(-1/3),
\]
or
\[
-\frac{13}{6} \zeta(3) = P_3(-1/3)-2P_3(1/3).
\]
\end{proof}

\begin{prop}\label{prop:p_adic_ident}
We have
\[
\Lip_3(-3)-2\Lip_3(3) = -\frac{13}{6} \zetap(3).
\]
\end{prop}

\begin{proof}
We consider the element $[-3]_3-2[3]_3 + \frac{13}{6}[1]_3$ of $H^1(\mathcal{M}^{\bullet}_3(\Zb[1/6]))$ of \cite{deJeu95}. By Remark 5.2 of loc. cit., Lemma \ref{lemma:complex_ident} implies that this element has trivial regulator. By Theorem 3.15 of loc. cit., the regulator is injective, so this element is zero in $H^1(\mathcal{M}^{\bullet}_3(\Zb[1/6]))$.

By \cite[Proposition 7.14]{BdJ04}, the syntomic regulator on $H^1(\mathcal{M}^{\bullet}_3(\Zb[1/6]))$ is a multiple of the ``single-valued'' $p$-adic trilogarithm $L_3$, defined as
\[
L_3(z) \colonequals \Lip_3(z)-\Lip_2(z)\logp(z) + \frac{1}{2} \Lip_1(z) \logp(z)^2.
\]
Therefore, we get the identity
\[
L_3(-3)-2L_3(3) + \frac{13}{6} L_3(1) = 0.
\]

Next, note that $\Delta'(\Liu_2(-3)-2\Liu_2(3))=0$ by a simple computation with Proposition \ref{prop:polylog_coprod}, hence $\Liu_2(-3)=2\Liu_2(3)$ because $E_2(\Zb[1/6])=K_3(\Zb[1/6])_{\Qb}=0$. Applying $\mathrm{per}_{p}$, we get
\[
\Lip_3(-3)=2\Lip_2(3).
\]

Next, note that
\begin{eqnarray*}
L_3(-3)-2L_3(3) &=&
\left(\Lip_3(-3)-\Lip_2(-3)\logp(-3) + \frac{1}{2} \Lip_1(-3) \logp(-3)^2\right)
\\
&& -
2\left(\Lip_3(3)-\Lip_2(3)\logp(3) + \frac{1}{2} \Lip_1(3) \logp(3)^2\right)
\\
&=&
\Lip_3(-3)-2\Lip_3(3) +\logp{3}\left(2\Lip_2(3)-\Lip_2(-3))\right)\\
&& +\frac{1}{2}\logp(3)^2\left(-\logp(4)+2\logp(-2)\right)
\\
&=&
\Lip_3(-3)-2\Lip_3(3).
\end{eqnarray*}

As well, since $\logp(1)=0$, we find that $L_3(1)=\zetap(3)$. Combining this with the above, we get the identity
\[
\Lip_3(-3)-2\Lip_3(3) = -\frac{13}{6} \zetap(3).
\]

\end{proof}

\bibliography{MTMUEIII_Refs}
\bibliographystyle{alpha}

\vfill

 \Small\textsc{D.C. Department of Mathematics, UC Berkeley, Evans Hall, Berkeley, CA 94720, USA}. {Email address:} \texttt{corwind@alum.mit.edu}

\Small\textsc{I.D. Department of Mathematics, Ben Gurion University of the Negev, Deichmann Building for Mathematics, Be'er Sheva 8410501, Israel}. {Email address:} \texttt{ishaidc@gmail.com}

\end{document}